\documentclass{article}
\usepackage[LGR,T1]{fontenc}
\usepackage[latin9]{inputenc}
\usepackage{color}
\usepackage{amsmath}
\usepackage{amsthm}
\usepackage{amssymb}
\usepackage{graphicx}
\usepackage[unicode=true,
 bookmarks=true,bookmarksnumbered=true,bookmarksopen=true,bookmarksopenlevel=3,
 breaklinks=false,pdfborder={0 0 1},backref=false,colorlinks=true]
 {hyperref}
\hypersetup{
 pdfauthor={Lei Yu},
 pdfborderstyle=,pdfborderstyle={},pdfborderstyle={},pdfborderstyle={},pdfborderstyle={},pdfborderstyle={},pdfborderstyle={},pdfborderstyle={},pdfborderstyle={},pdfborderstyle={},pdfborderstyle={},pdfborderstyle={},pdfborderstyle={},pdfborderstyle={},pdfborderstyle={},pdfborderstyle={},pdfborderstyle={},pdfpagelayout=OneColumn,pdfnewwindow=true,pdfstartview=XYZ,plainpages=false,linkcolor=blue,urlcolor=blue,citecolor=red,anchorcolor=blue,linkcolor=blue,urlcolor=blue,citecolor=red,anchorcolor=blue}

\makeatletter
\theoremstyle{plain}
\newtheorem{thm}{\protect\theoremname}
\theoremstyle{remark}
\newtheorem{rem}[]{\protect\remarkname}
\theoremstyle{plain}
\newtheorem{cor}[]{\protect\corollaryname}
\theoremstyle{plain}
\newtheorem{lem}[]{\protect\lemmaname}
\theoremstyle{definition}
\newtheorem{defn}[]{\protect\definitionname}
\theoremstyle{plain}
\newtheorem{prop}[]{\protect\propositionname}

\usepackage{fullpage}

\usepackage{color}
\usepackage{array}
\usepackage{url}

\usepackage{latexsym}
\usepackage{bm}
\usepackage{wrapfig}
\usepackage{fancybox}

\usepackage{bm}
\usepackage{epstopdf}
\usepackage{float}
\usepackage{subfloat}
\usepackage{array}
\usepackage{tabularx}
\usepackage{multirow}
\usepackage{tikz}
\usepackage{algorithmic}

\DeclareMathOperator*{\esssup}{\mathrm{ess\,sup}}
\DeclareMathOperator*{\essinf}{\mathrm{ess\,inf}}

\tikzstyle{arw}=[->,>=latex]
\tikzstyle{node}=[draw,rectangle,rounded corners, minimum width=1cm,minimum height =.75 cm]

\usepackage{color}

\usepackage{adjustbox}

\usepackage{algorithmic}

\providecommand{\corollaryname}{Corollary}
\providecommand{\definitionname}{Definition}
\providecommand{\lemmaname}{Lemma}
\providecommand{\remarkname}{Remark}
\providecommand{\theoremname}{Theorem}
\usepackage{amsfonts}
\usepackage{amsopn}
\usepackage{bbm}

 \def\clap#1{\hbox to 0pt{\hss#1\hss}}

\allowdisplaybreaks
\providecommand{\keywords}[1]{\textbf{{Index terms---}} #1}

\providecommand{\corollaryname}{Corollary}
\providecommand{\definitionname}{Definition}
\providecommand{\lemmaname}{Lemma}
\providecommand{\propositionname}{Proposition}
\providecommand{\remarkname}{Remark}
\providecommand{\theoremname}{Theorem}

\makeatother

\providecommand{\corollaryname}{Corollary}
\providecommand{\definitionname}{Definition}
\providecommand{\lemmaname}{Lemma}
\providecommand{\propositionname}{Proposition}
\providecommand{\remarkname}{Remark}
\providecommand{\theoremname}{Theorem}

\begin{document}
\title{On Conditional Correlations }

\author{Lei Yu \thanks{ Department of Electrical and Computer Engineering, National University
of Singapore, Singapore (Email: leiyu@nus.edu.sg). } }

\maketitle
\begin{abstract}
The Pearson correlation, correlation ratio, and maximal correlation
have been well-studied in the literature. In this paper, we study
the conditional versions of these quantities. We extend the most important
properties of the unconditional versions to the conditional versions,
and also derive some new properties. Based on the conditional maximal
correlation, we define an information-correlation function of two
arbitrary random variables, and use it to derive an impossibility
result for the problem of the non-interactive simulation of random
variables. 
\end{abstract}
\keywords{Correlation coefficient, correlation ratio, maximal correlation, information-correlation function, non-interactive
simulation} 

\section{Introduction}

In the literature, there are various measures available to quantify
the strength of the dependence between two random variables. These
include the Pearson correlation coefficient, the correlation ratio,
the maximal correlation coefficient, etc. The Pearson correlation
coefficient is a well-known measure that quantifies the linear dependence
between two real-valued random variables. For real-valued random variables
$X$ and $Y$, it is defined as 
\[
\rho(X;Y)=\left\{ \begin{array}{ll}
\frac{\mathrm{cov}(X,Y)}{\sqrt{\mathrm{var}(X)}\sqrt{\mathrm{var}(Y)}}, & \mathrm{var}(X)\mathrm{var}(Y)>0,\\
0, & \mathrm{var}(X)\mathrm{var}(Y)=0.
\end{array}\right.
\]
The correlation ratio was introduced by Pearson (see e.g. \cite{cramer2016mathematical}),
and studied by Rényi \cite{renyi1959measures,renyi1959new}. For a
real-valued random variable $X$ and a random variable $Y$ defined
on an arbitrary Borel-measurable space, the correlation ratio of $X$
on $Y$ is defined by 
\[
{\displaystyle \theta(X;Y)=\sup_{g}\rho(X;g(Y))},
\]
where the supremum is taken over all Borel-measurable real-valued
functions $g(y)$ such that $\mathrm{var}(g(Y))<\infty$. Rényi \cite{renyi1959measures,renyi1959new}
showed that 
\begin{align*}
\theta(X;Y) & =\sqrt{\frac{\mathrm{var}(\mathbb{E}[X|Y])}{\mathrm{var}(X)}}=\sqrt{1-\frac{\mathbb{E}[\mathrm{var}(X|Y)]}{\mathrm{var}(X)}}.
\end{align*}
Another related measure of dependence is the \emph{Hirschfeld-Gebelein-Rényi
maximal correlation} (or simply \emph{maximal correlation}), which
quantifies the maximum possible (Pearson) correlation between square
integrable real-valued random variables that are respectively generated
by each of two random variables. For two random variables $X$ and
$Y$ defined on arbitrary Borel-measurable spaces, the maximal correlation
between $X$ and $Y$ is defined by 
\[
{\displaystyle \rho_{\mathrm{m}}(X;Y)=\sup_{f,g}\rho(f(X);g(Y))},
\]
where the supremum is taken over all Borel-measurable real-valued
functions $f(x),g(y)$ such that $\mathrm{var}(f(X)),$ $\mathrm{var}(g(Y))<\infty$.
This measure was first introduced by Hirschfeld \cite{hirschfeld1935connection}
and Gebelein \cite{gebelein1941statistische}, then studied by Rényi
\cite{renyi1959measures}. Recently it has been exploited in studying
some information-theoretic problems, such as measuring non-local correlations
\cite{beigi2015monotone}, maximal correlation secrecy \cite{li2018maximal},
deriving converse results for distributed communication \cite{yu2016distortion},
etc. Furthermore, the maximal correlation is also related to the Gács-Körner
or Wyner common information \cite{gacs1973common,witsenhausen1975sequences}.
The Gács-Körner common information is strictly positive, if and only
if the maximal correlation is equal to $1$. The Wyner common information
is strictly positive, if and only if the maximal correlation is strictly
positive.

In this paper, we extend the Pearson correlation, the correlation
ratio, and the maximal correlation to their conditional versions.
We investigate various properties of these correlations. We also introduce
an information-correlation function of two arbitrary random variables,
and use it to derive an impossibility result for the problem of the
non-interactive simulation of random variables.

\section{Definition}

Let $(\Omega,\Sigma,\mathbb{P})$ be a probability space. Let $\left(X,Y,Z,U\right):\left(\Omega,\Sigma\right)\to\left(\mathbb{R}^{4},\mathcal{B}\left(\mathbb{R}^{4}\right)\right)$
be a real-valued random vector, where $\mathcal{B}\left(\mathbb{R}^{4}\right)$
denotes the Borel $\sigma$-algebra on $\mathbb{R}^{4}$. For a random
variable (or random vector) $W$, we denote the {\em probability
distribution} as $P_{W}$, i.e., $P_{W}:=\mathbb{P}\circ W^{-1}$.
If $W$ is discrete, then we use $P_{W}$ to denote the probability
mass function (pmf). If $W$ is absolutely continuous (the distribution
is absolutely continuous respect to the Lebesgue measure), then we
use $p_{W}$ to denote the probability density function (pdf).

In the following, we define several conditional correlations, including
the conditional (Pearson) correlation, the conditional correlation
ratio, and the conditional maximal correlation. \begin{defn} The
\emph{conditional (Pearson) correlation}\footnote{Here $U$ does not need to be real-valued. But for brevity, we assume
it is. Similarly, in the following, $\left(Y,U\right)$ does not need
to be real-valued in the definition of conditional correlation ratio,
and $\left(X,Y,U\right)$ does not need to be real-valued in the definition
of the conditional maximal correlation. } of $X$ and $Y$ given $U$ is defined by 
\[
\rho(X;Y|U)=\left\{ \begin{array}{ll}
\frac{\mathbb{E}[\mathrm{cov}(X,Y|U)]}{\sqrt{\mathbb{E}[\mathrm{var}(X|U)]}\sqrt{\mathbb{E}[\mathrm{var}(Y|U)]}}, & \mathbb{E}[\mathrm{var}(X|U)]\mathbb{E}[\mathrm{var}(Y|U)]>0,\\
0, & \mathbb{E}[\mathrm{var}(X|U)]\mathbb{E}[\mathrm{var}(Y|U)]=0.
\end{array}\right.
\]
\end{defn} \begin{defn} The \emph{conditional correlation ratio}
of $X$ on $Y$ given $U$ is defined by 
\begin{equation}
\theta(X;Y|U)=\sup_{g}\rho(X;g(Y,U)|U),\label{eq:}
\end{equation}
where the supremum is taken over all Borel-measurable real-valued
functions $g(y,u)$ such that $\;\mathbb{E}[\mathrm{var}(g(Y,U)|U)]<\infty$.
\end{defn} \begin{defn} \label{def:For-any-random}The \emph{conditional
maximal correlation} of $X$ and $Y$ given $U$ is defined by 
\[
\rho_{\mathrm{m}}(X;Y|U)=\sup_{f,g}\rho(f(X,U);g(Y,U)|U),
\]
where the supremum is taken over all Borel-measurable real-valued
functions $f(x,u),g(y,u)$ such that $\mathbb{E}[\mathrm{var}(f(X,U)|U)]$,
$\mathbb{E}[\mathrm{var}(g(Y,U)|U)]<\infty$. \end{defn} \begin{rem}
If $U$ is degenerate, then these three conditional correlations reduce
to their unconditional versions. \end{rem} \begin{rem} Note that
$\rho(X;Y|U)=\rho(Y;X|U)$ and $\rho_{\mathrm{m}}(X;Y|U)=\rho_{\mathrm{m}}(Y;X|U)$,
but in general $\theta(X;Y|U)\neq\theta(Y;X|U)$. That is, the conditional
correlation and the conditional maximal correlation are symmetric,
but the conditional correlation ratio is not. \end{rem} By the definitions,
it is easy to verify that 
\begin{equation}
\rho_{\mathrm{m}}(X;Y|U)=\sup_{f}\theta(f(X,U);Y|U),\label{eq:-27}
\end{equation}
where the supremum is taken over all Borel-measurable real-valued
functions $f(x,u)$ such that $\mathbb{E}[\mathrm{var}(f(X,U)|U)]<\infty$.

Note that the unconditional versions of correlation coefficient, correlation
ratio, and maximal correlation have been studied extensively in the
literature; see \cite{renyi1959measures,renyi1959new}. \textcolor{blue}{The
conditional version of maximal correlation was first introduced by
Ardestanizadeh }\textcolor{blue}{\emph{et al.}}\textcolor{blue}{{} \cite{ardestanizadeh2012linear}.
They studied the conditional maximal correlation of Gaussian random
variables, and showed that for this case, the conditional maximal
correlation is equal to the conditional Pearson correlation. They
applied this property to derive upper bounds for the sum-capacity
of a Gaussian multi-access channel (with linear feedback).  Beigi
and Gohari \cite{beigi2015monotone} applied the conditional maximal
correlation to study the problem of non-local correlations in a bipartite
quantum system (which is modeled as a }\textcolor{blue}{\emph{no-signaling
box}}\textcolor{blue}{). For such a no-signaling box, a sub-tensorization
property of the conditional maximal correlation was proven \cite[Corollary 6]{beigi2015monotone}.
The sub-tensorization (or tensorization) property is rather useful
in bounding the correlation between two random vectors, especially
when the two random vectors consist of a large number of i.i.d. pairs
of components. This is because, due to the sub-tensorization (or tensorization)
property, the resulting bound is independent of the number of components
in the vectors, and hence it is non-trivial even when the number of
components tends to infinity.  In \cite{beigi2016duality} Beigi
and Gohari studied the relationship between the conditional maximal
correlation and the conditional hypercontractivity. They also introduced
a general principle to obtain new measures from additivity measures
such that the new measures have both tensorization and data processing
properties.  
 In this paper, we study various properties of the conditional versions
of Pearson correlation, correlation ratio, and maximal correlation
of }\textcolor{blue}{\emph{arbitrary}}\textcolor{blue}{{} random variables
(not restricted to be Gaussian or discrete).} In order to state our
results clearly, we define the various correlations conditioned on
a given event as follows. \begin{defn} Given an event $\mathcal{A}$,
denote the conditional distribution of $(X,Y)$ given $\mathcal{A}$
as $\mathrm{P}_{X,Y|\mathcal{A}}$. Assume $(X',Y')$ is a pair of
random variables satisfying $(X',Y')\sim\mathrm{P}_{X,Y|\mathcal{A}}$.
Then we define $\kappa(X;Y|\mathcal{A}):=\kappa(X';Y')$ as the \emph{event
conditional correlations} of $X$ and $Y$ given $\mathcal{A}$, where
$\kappa\in\left\{ \rho,\theta,\rho_{\mathrm{m}}\right\} $ and $\kappa(X';Y')$
denotes the corresponding unconditional correlation of $X'$ and $Y'$.
\end{defn} Obviously, these event conditional correlations are special
cases of the corresponding conditional correlations. Moreover, if
the distribution of $(X',Y')$ is the same as the conditional distribution
of $(X,Y)$ given $U=u$, then the unconditional correlations of $(X',Y')$
respectively equal the corresponding event conditional correlations
of $(X,Y)$ given $U=u$, i.e., $\kappa(X';Y')=\kappa(X;Y|U=u)$ where
$\kappa\in\left\{ \rho,\theta,\rho_{\mathrm{m}}\right\} $. Moreover,
if the distribution of $U$ satisfies $\mathbb{P}\left(U=u\right)=1$
for some $u$, then the conditional correlations of $(X,Y)$ given
$U$ respectively equal the corresponding event conditional correlations
of $(X,Y)$ given $U=u$, i.e., $\kappa(X;Y|U)=\kappa(X;Y|U=u)$ where
$\kappa\in\left\{ \rho,\theta,\rho_{\mathrm{m}}\right\} $.

\section{Properties}

\subsection{Basic Properties: Other Characterizations, Continuity, and Concavity }

In this subsection, we provide other characterizations for the conditional
correlation ratio and conditional maximal correlation, and then study
continuity (or discontinuity) and concavity of the conditional maximal
correlation. First by definition, we have the following basic properties. 
\begin{thm}
For any random variables $X,Y,Z,U$, we have that 
\begin{align*}
\theta(X;Y,Z|U) & \geq\theta(X;Y|U);\\
\rho_{\mathrm{m}}(X;Y,Z|U) & \geq\rho_{\mathrm{m}}(X;Y|U).
\end{align*}
\end{thm}

Next we characterize the conditional correlation ratio and conditional
maximal correlation by ratios of variances. 
\begin{thm}
\label{lem:For-any-random}(Characterization by the ratio of variances).
For any random variables $X,Y,Z,U$, we have that 
\begin{align}
\theta(X;Y|U) & =\sqrt{\frac{\mathbb{E}[\mathrm{var}(\mathbb{E}[X|Y,U]|U)]}{\mathbb{E}[\mathrm{var}(X|U)]}}\nonumber \\
 & =\sqrt{1-\frac{\mathbb{E}[\mathrm{var}(X|Y,U)]}{\mathbb{E}[\mathrm{var}(X|U)]}};\label{eq:-12}\\
\rho_{\mathrm{m}}(X;Y|U) & =\sup_{f}{\sqrt{\frac{\mathbb{E}[\mathrm{var}(\mathbb{E}[f(X,U)|Y,U]|U)]}{\mathbb{E}[\mathrm{var}(f(X,U)|U)]}}}\nonumber \\
 & =\sup_{f}{\sqrt{1-\frac{\mathbb{E}[\mathrm{var}(f(X,U)|Y,U)]}{\mathbb{E}[\mathrm{var}(f(X,U)|U)]}}}.\label{eq:-13}
\end{align}
\end{thm}

\begin{rem}
The correlation ratio is also closely related to the Minimum Mean
Square Error (MMSE). The optimal MMSE estimator is $\mathbb{E}[X|Y,U]$,
hence the variance of the MMSE for estimating $X$ given $(Y,U)$
is $\mathrm{mmse}(X|Y,U)=\mathbb{E}(X-\mathbb{E}[X|Y,U])^{2}=\mathbb{E}[\mathrm{var}(X|Y,U)]=\mathbb{E}[\mathrm{var}(X|U)](1-\theta^{2}(X;Y|U)).$ 
\end{rem}

\begin{rem}
Equation \eqref{eq:-13} was first proven in \cite[Lemma 15]{beigi2016duality}. 
\end{rem}

The unconditional version of Theorem \ref{lem:For-any-random} was
proven by Rényi \cite{renyi1959measures}. Theorem \ref{lem:For-any-random}
can be proven similarly as the unconditional versions in \cite{renyi1959measures}.
Hence the proof is omitted here. Next we characterize conditional
correlations by event conditional correlations. 
\begin{thm}
\label{thm:Alternative-characterization}(Characterization by event
conditional correlations). For any random variables $X,Y,U,$ 
\begin{align}
\rho(X;Y|U) & \leq\esssup_{u}\rho(X;Y|U=u),\label{eq:-3}\\
\essinf_{u}\theta(X;Y|U=u)\leq\theta(X;Y|U) & \leq\esssup_{u}\theta(X;Y|U=u),\label{eq:-5}\\
\rho_{\mathrm{m}}(X;Y|U) & =\esssup_{u}\rho_{\mathrm{m}}(X;Y|U=u),\label{eq:-4}
\end{align}
where $\esssup_{u}f(u):=\inf\left\{ \lambda:\mathbb{P}\left(f(U)>\lambda\right)=0\right\} $
and $\essinf_{u}f(u):=\sup\left\{ \lambda:\mathbb{P}\left(f(U)<\lambda\right)=0\right\} $
respectively denote the essential supremum and the essential infimum
of $f$. 
\end{thm}

\begin{rem}
It is worth noting that $\rho(X;Y|U)\geq\essinf_{u}\rho(X;Y|U=u)$
does not hold in general. This can be seen from the following example.
Assume $(a,b,\eta)$ are three numbers such that $0<\eta\leq1,0<a<b$.
Suppose that $(W,Z)$ is a pair of random variables such that $\mathrm{var}(W)=a,\mathrm{var}(Z)=b$
and $\rho(W;Z)=\eta$. (It is obvious that there are many random variable
pairs satisfying the conditions.) Denote the distribution of $(W,Z)$
as $\mathrm{P}_{W,Z}$. Now we consider a triple of random variables
$(X,Y,U)$ such that $P_{U}(0)=P_{U}(1)=\frac{1}{2}$ and $(X,Y)|U=0\sim\mathrm{P}_{W,Z}$
and $(X,Y)|U=1\sim\mathrm{P}_{Z,W}$. Then we have $\rho(X;Y|U=0)=\rho(X;Y|U=1)=\eta$.
Hence $\essinf_{u}\rho(X;Y|U=u)=\eta$. However, $\rho(X;Y|U)=\frac{2\eta\sqrt{ab}}{a+b}<\eta$.
Hence $\rho(X;Y|U)<\essinf_{u}\rho(X;Y|U=u)$ for this example. 
\end{rem}

\begin{rem}
If $U$ is a discrete random variable, then 
\begin{equation}
\rho_{\mathrm{m}}(X;Y|U)=\sup_{u:P_{U}(u)>0}\rho_{\mathrm{m}}(X;Y|U=u),\label{eq:-4-1}
\end{equation}
where $P_{U}$ denotes the pmf of $U$. Beigi and Gohari \cite{beigi2015monotone,beigi2016duality}
defined the conditional maximal correlation via \eqref{eq:-4-1}.
Theorem \ref{thm:Alternative-characterization} implies the equivalence
between the conditional maximal correlation defined by us and that
defined by Beigi and Gohari. 
\end{rem}

\begin{rem}
If $U$ is an absolutely continuous random variable, then 
\[
\rho_{\mathrm{m}}(X;Y|U)=\inf_{q_{U}:q_{U}=p_{U}\mathrm{a.e.}}\sup_{u:q_{U}(u)>0}\rho_{\mathrm{m}}(X;Y|U=u),
\]
where $p_{U}$ denotes the pdf of $U$ and $q_{U}$ denotes another
pdf on the same space. 
\end{rem}

\begin{proof}
We first prove \eqref{eq:-3}. Denote $\mathcal{A}_{\lambda}:=\left\{ u:\rho(X;Y|U=u)>\lambda\right\} $
and $\lambda^{*}:=\inf\left\{ \lambda:\mathrm{P}_{U}\left(\mathcal{A}_{\lambda}\right)=0\right\} $.
Hence $P_{U}\left(\mathcal{A}_{\lambda}\right)=0$ for any $\lambda>\lambda^{*}$;
and $\mathrm{P}_{U}\left(\mathcal{A}_{\lambda}\right)>0$ for any
$\lambda<\lambda^{*}$. It means that $\lambda^{*}=\esssup_{u}\rho(X;Y|U=u)$.
Therefore, to show \eqref{eq:-3}, we only need to show $\rho(X;Y|U)\leq\lambda^{*}$.
To this end, we upper bound $\rho(X;Y|U)$ as follows. 
\begin{align}
\rho(X;Y|U) & =\frac{\mathbb{E}[\mathrm{cov}(X,Y|U)]}{\sqrt{\mathbb{E}[\mathrm{var}(X|U)]}\sqrt{\mathbb{E}[\mathrm{var}(Y|U)]}}\label{eq:-28}\\
 & \le\frac{\mathbb{E}[\mathrm{cov}(X,Y|U)]}{\mathbb{E}\sqrt{\mathrm{var}(X|U)\mathrm{var}(Y|U)}}\label{eq:-6}\\
 & =\inf_{\lambda>\lambda^{*}}\frac{\mathbb{E}[\mathrm{cov}(X,Y|U)\cdot1\left\{ U\in\mathbb{R}\backslash\mathcal{A}_{\lambda}\right\} ]}{\mathbb{E}\left[\sqrt{\mathrm{var}(X|U)\mathrm{var}(Y|U)}\cdot1\left\{ U\in\mathbb{R}\backslash\mathcal{A}_{\lambda}\right\} \right]}\label{eq:-29}\\
 & \le\inf_{\lambda>\lambda^{*}}\sup_{u\in\mathbb{R}\backslash\mathcal{A}_{\lambda}}\rho(X;Y|U=u)\nonumber \\
 & \leq\inf_{\lambda>\lambda^{*}}\lambda=\lambda^{*},\label{eq:-30}
\end{align}
where \eqref{eq:-6} follows by the Cauchy-Schwarz inequality, and
\eqref{eq:-29} follows from \cite[Theorem 15.2 (v)]{billingsley2008probability}
and the fact $\mathrm{P}_{U}\left(\mathcal{A}_{\lambda}\right)=0$
for any $\lambda>\lambda^{*}$.

By using the relationship \eqref{eq:-12} and by derivations similar
as \eqref{eq:-28}-\eqref{eq:-30}, it is easy to obtain \eqref{eq:-5}.

Finally, we prove \eqref{eq:-4}. Similarly as in the proof above,
we denote $\mathcal{A}_{\lambda}:=\left\{ u:\rho_{\mathrm{m}}(X;Y|U=u)>\lambda\right\} $
and $\lambda^{*}:=\inf\left\{ \lambda:\mathrm{P}_{U}\left(\mathcal{A}_{\lambda}\right)=0\right\} $.
Hence $\mathrm{P}_{U}\left(\mathcal{A}_{\lambda}\right)=0$ for any
$\lambda>\lambda^{*}$; $\mathrm{P}_{U}\left(\mathcal{A}_{\lambda}\right)>0$
for any $\lambda<\lambda^{*}$; and $\lambda^{*}=\esssup_{u}\rho_{\mathrm{m}}(X;Y|U=u)$.
Therefore, to prove \eqref{eq:-4}, we only need to show $\rho_{\mathrm{m}}(X;Y|U)=\lambda^{*}$.
On one hand, by derivations similar as \eqref{eq:-28}-\eqref{eq:-30},
we can upper bound $\rho_{\mathrm{m}}(X;Y|U)$ as follows. 
\begin{align*}
\rho_{\mathrm{m}}(X;Y|U) & =\sup_{f}{\displaystyle \sqrt{\frac{\mathbb{E}[\mathrm{var}(\mathbb{E}[f(X,U)|Y,U]|U)]}{\mathbb{E}[\mathrm{var}(f(X,U)|U)]}}}\\
 & =\sup_{f}\inf_{\lambda>\lambda^{*}}{\displaystyle \sqrt{\frac{\mathbb{E}[\mathrm{var}(\mathbb{E}[f(X,U)|Y,U]|U)\cdot1\left\{ U\in\mathbb{R}\backslash\mathcal{A}_{\lambda}\right\} ]}{\mathbb{E}[\mathrm{var}(f(X,U)|U)\cdot1\left\{ U\in\mathbb{R}\backslash\mathcal{A}_{\lambda}\right\} ]}}}\\
 & \leq\sup_{f}\inf_{\lambda>\lambda^{*}}\sup_{u\in\mathbb{R}\backslash\mathcal{A}_{\lambda}}{\displaystyle \sqrt{\frac{\mathbb{E}[\mathrm{var}(\mathbb{E}[f(X,U)|Y,U]|U=u)]}{\mathbb{E}[\mathrm{var}(f(X,U)|U=u)]}}}\\
 & \leq\inf_{\lambda>\lambda^{*}}\sup_{u\in\mathbb{R}\backslash\mathcal{A}_{\lambda}}\sup_{f}{\displaystyle \sqrt{\frac{\mathbb{E}[\mathrm{var}(\mathbb{E}[f(X,U)|Y,U]|U=u)]}{\mathbb{E}[\mathrm{var}(f(X,U)|U=u)]}}}\\
 & =\inf_{\lambda>\lambda^{*}}\sup_{u\in\mathbb{R}\backslash\mathcal{A}_{\lambda}}\rho_{\mathrm{m}}(X;Y|U=u)\\
 & \leq\inf_{\lambda>\lambda^{*}}\lambda=\lambda^{*}.
\end{align*}

On the other hand, we assume that $\widetilde{f}(x,u)$ is a function
such that ${\displaystyle \sqrt{\frac{\mathrm{var}(\mathbb{E}[\widetilde{f}(X,U)|Y,U=u]|U=u)}{\mathrm{var}(\widetilde{f}(X,U)|U=u)}}}\geq\alpha\rho_{\mathrm{m}}(X;Y|U=u)$
for each $u\in\mathcal{A}_{\lambda}$, where $\lambda<\lambda^{*}$
and $0<\alpha<1$. The existence of $\widetilde{f}(x,u)$ follows
from the definition of $\rho_{\mathrm{m}}(X;Y|U=u)$. According to
the definition of $\mathcal{A}_{\lambda}$, we have that $\mathrm{P}_{U}\left(\mathcal{A}_{\lambda}\right)>0$,
and for each $u\in\mathcal{A}_{\lambda}$, 
\begin{equation}
\frac{\mathrm{var}(\mathbb{E}[\widetilde{f}(X,U)|Y,U=u]|U=u)}{\mathrm{var}(\widetilde{f}(X,U)|U=u)}\geq\left(\alpha\lambda\right)^{2}.\label{eq:-1}
\end{equation}
Set $f(x,u)=\widetilde{f}(x,u)\cdot1\left\{ u\in\mathcal{A}_{\lambda}\right\} $.
Then 
\begin{align}
\rho_{\mathrm{m}}(X;Y|U) & \geq{\displaystyle \sqrt{\frac{\mathbb{E}[\mathrm{var}(\mathbb{E}[f(X,U)|Y,U]|U)]}{\mathbb{E}[\mathrm{var}(f(X,U)|U)]}}}\nonumber \\
 & ={\displaystyle \sqrt{\frac{\mathbb{E}[\mathrm{var}(\mathbb{E}[\widetilde{f}(X,U)|Y,U]|U)\cdot1\left\{ U\in\mathcal{A}_{\lambda}\right\} ]}{\mathbb{E}[\mathrm{var}(\widetilde{f}(X,U)|U)\cdot1\left\{ U\in\mathcal{A}_{\lambda}\right\} ]}}}\nonumber \\
 & \geq{\displaystyle \sqrt{\inf_{u\in\mathcal{A}_{\lambda}}\frac{\mathrm{var}(\mathbb{E}[\widetilde{f}(X,U)|Y,U=u]|U=u)}{\mathrm{var}(\widetilde{f}(X,U)|U=u)}}}\nonumber \\
 & \geq\alpha\lambda,\label{eq:-24}
\end{align}
where \eqref{eq:-24} follows from \eqref{eq:-1}. Since $\lambda<\lambda^{*}$
and $0<\alpha<1$ are arbitrary, we have $\rho_{\mathrm{m}}(X;Y|U)\geq\lambda^{*}$.

Combining the two points above, we have $\rho_{\mathrm{m}}(X;Y|U)=\lambda^{*}$. 
\end{proof}
For discrete $\left(X,Y\right)$ with finite supports, without loss
of generality, the supports of $X$ and $Y$ are assumed to be $\left\{ 1,2,...,m\right\} $
and $\left\{ 1,2,...,n\right\} $, respectively. For this case, denote
$\lambda_{2}(u)$ as the second largest singular value of the matrix
$Q_{u}$ with entries 
\[
Q_{u}(x,y):=\frac{P(x,y|u)}{\sqrt{P(x|u)P(y|u)}}=\frac{P(x,y,u)}{\sqrt{P(x,u)P(y,u)}}.
\]
For absolutely-continuous $X,Y$, denote $\lambda_{2}(u)$ as the
second largest singular value of the bivariate function $\frac{p(x,y|u)}{\sqrt{p(x|u)p(y|u)}}$,
where $p(x,y|u)$ denotes a conditional pdf of $(X,Y)$ respect to
$U$. Then we have the following singular value characterization of
conditional maximal correlation. 
\begin{thm}
\label{lem:Singular-value-characterization}(Singular value characterization).
Assume $X,Y$ are discrete random variables with finite supports,
or absolutely-continuous random variables such that$\int_{x,y}\left(\frac{p(x,y|u)}{\sqrt{p(x|u)p(y|u)}}\right)^{2}dxdy<\infty$
a.s. Then 
\begin{equation}
\rho_{\mathrm{m}}(X;Y|U)=\esssup_{u}\lambda_{2}(u).\label{eq:-4-2}
\end{equation}
\end{thm}

\begin{rem}
This property is consistent with the one of the unconditional version
by setting $U$ to a constant, i.e., $\rho_{\mathrm{m}}(X;Y)=\lambda_{2}.$ 
\end{rem}

\begin{proof}
The unconditional version of this theorem was proven in \cite{witsenhausen1975sequences}.
That is, for discrete $X,Y$ with finite supports, $\rho_{\mathrm{m}}(X;Y)$
equals the second largest singular value $\lambda_{2}$ of the matrix
$Q$ with entries $Q(x,y):=\frac{P(x,y)}{\sqrt{P(x)P(y)}}$; for absolutely-continuous
$X,Y$ such that $\int_{x,y}\left(\frac{p(x,y)}{\sqrt{p(x)p(y)}}\right)^{2}dxdy<\infty$
with $p(x,y)$ denoting a pdf of $(X,Y)$, $\rho_{\mathrm{m}}(X;Y)$
equals the second largest singular value $\lambda_{2}$ of the bivariate
function $\frac{p(x,y)}{\sqrt{p(x)p(y)}}$. Combining this with Theorem
\ref{thm:Alternative-characterization}, we have \eqref{eq:-4-2}. 
\end{proof}
Note that, $\rho_{\mathrm{m}}(X;Y|U)$ is a mapping that maps a distribution
$\mathrm{P}_{X,Y,U}$ to a real number in $[0,1]$. Now we study the
concavity of such a mapping. 
\begin{cor}
(Concavity). Given $\mathrm{P}_{X,Y|U}$, $\rho_{\mathrm{m}}(X;Y|U)$
is concave in $\mathrm{P}_{U}$. That is, for any distributions $\mathrm{P}_{U}$
and $\mathrm{Q}_{U}$, and any $\lambda\in[0,1]$, $\rho_{\mathrm{m}}^{(\left((1-\lambda)\mathrm{P}_{U}+\lambda\mathrm{Q}_{U}\right)\mathrm{P}_{X,Y|U})}(X;Y|U)\geq(1-\lambda)\rho_{\mathrm{m}}^{(\mathrm{P}_{U}\mathrm{P}_{X,Y|U})}(X;Y|U)+\lambda\rho_{\mathrm{m}}^{(\mathrm{Q}_{U}\mathrm{P}_{X,Y|U})}(X;Y|U)$,
where $\rho_{\mathrm{m}}^{(Q_{X,Y,U})}(X;Y|U)$ denotes the conditional
maximal correlation of $X$ and $Y$ given $U$ under distribution
$Q_{X,Y,U}$. 
\end{cor}

\begin{proof}
This theorem directly follows from the characterization in \eqref{eq:-4}. 
\end{proof}
For a discrete random variable, the distribution is uniquely determined
by its pmf. Therefore, for discrete random variables $\left(X,Y,U\right)$,
$\rho_{\mathrm{m}}(X;Y|U)$ can be also seen as a mapping that maps
a pmf $P_{X,Y,U}$ to a real number in $[0,1]$. Assume $\mathcal{X},\mathcal{Y},\mathcal{U}\subset\mathbb{R}$
are three finite sets. Denote $\mathcal{P}\left(\mathcal{X}\times\mathcal{Y}\times\mathcal{U}\right)$
as the set of pmfs defined on $\mathcal{X}\times\mathcal{Y}\times\mathcal{U}$
(i.e., the $\left|\mathcal{X}\right|\left|\mathcal{Y}\right|\left|\mathcal{U}\right|-1$
dimensional probability simplex). Consider $\rho_{\mathrm{m}}(X;Y|U)$
as a mapping $\rho_{\mathrm{m}}(X;Y|U):\mathcal{P}\left(\mathcal{X}\times\mathcal{Y}\times\mathcal{U}\right)\to[0,1]$.
Now we study the continuity (or discontinuity) of such a mapping. 
\begin{cor}
(Continuity and discontinuity). For finite sets $\mathcal{X},\mathcal{Y},\mathcal{U}\subset\mathbb{R}$,
$\rho_{\mathrm{m}}(X;Y|U)$ is continuous (under the total variation
distance) on $\{P_{X,Y,U}\in\mathcal{P}\left(\mathcal{X}\times\mathcal{Y}\times\mathcal{U}\right):P_{U}(u)>0,\forall u\in\mathcal{U}\}$.
But in general, $\rho_{\mathrm{m}}(X;Y|U)$ is discontinuous at $P_{U}$
such that $P_{U}(u)=0,\exists u\in\mathcal{U}$. 
\end{cor}

\begin{proof}
For a pmf $P_{X,Y,U}\in\mathcal{P}\left(\mathcal{X}\times\mathcal{Y}\times\mathcal{U}\right)$,
$\rho_{\mathrm{m}}(X;Y|U)={\displaystyle \max_{u:P(u)>0}\lambda_{2}(u)}$.
On the other hand, singular values are continuous in the matrix (see
\cite[Corollary 8.6.2]{golub2012matrix}), hence $\lambda_{2}(u)$
is continuous in $P_{X,Y|U=u}$. Furthermore, since $P_{U}(u)>0,\forall u\in\mathcal{U}$,
$Q_{X,Y,U}\rightarrow P_{X,Y,U}$ in the total variation distance
sense implies $Q_{U}\rightarrow P_{U}$ and $Q_{X,Y|U=u}\rightarrow P_{X,Y|U=u},\forall u\in\mathcal{U}$.
Therefore, $\rho_{\mathrm{m}}^{(Q)}(X;Y|U)\rightarrow\rho_{\mathrm{m}}(X;Y|U)$,
where $\rho_{\mathrm{m}}^{(Q)}(X;Y|U)$ denotes the conditional maximal
correlation under distribution $Q_{X,Y,U}$. However, if there exists
$u_{0}\in\mathcal{U}$ such that $P_{U}(u_{0})=0$ and $\lambda_{2}(u_{0})>\max_{u:P(u)>0}\lambda_{2}(u)$.
Then letting $Q_{U}\rightarrow P_{U}$ in a direction such that $Q_{U}(u_{0})>0$
always holds, we have that $\rho_{\mathrm{m}}^{(Q)}(X;Y|U)\geq\lambda_{2}(u_{0})>\rho_{\mathrm{m}}(X;Y|U)$
always holds. This implies $\rho_{\mathrm{m}}(X;Y|U)$ is discontinuous
at $P_{X,Y,U}$. 
\end{proof}
For random variables $X,Y,U$, the conditional Gács-Körner common
information between $X$ and $Y$ given $U$ is defined as 
\begin{equation}
C_{\mathrm{GK}}(X;Y|U)=\sup_{\left(f,g\right):f(X,U)=g(Y,U)\textrm{ a.s.}}H(f(X,U)|U),\label{eq:GK}
\end{equation}
where $H(Z|U):=-\mathbb{E}\log P_{Z|U}(Z|U)$ denotes the conditional
entropy of $Z$ given $U$. If $U$ is degenerate, then $C_{\mathrm{GK}}(X;Y):=C_{\mathrm{GK}}(X;Y|U)$
is the unconditional version of Gács-Körner common information between
$X$ and $Y$ \cite{gacs1973common}.
\begin{thm}
\label{lem:relationship}1) For any random variables $X,Y,U$, 
\[
0\leq|\rho(X;Y|U)|\leq\theta(X;Y|U)\leq\rho_{\mathrm{m}}(X;Y|U)\leq1.
\]
2) Moreover, $\rho_{\mathrm{m}}(X;Y|U)=0$ if and only if $X$ and
$Y$ are conditionally independent given $U$. Furthermore, for discrete
random variables $X,Y,U$ with finite supports, $\rho_{\mathrm{m}}(X;Y|U)=1$
if and only if $C_{\mathrm{GK}}(X;Y|U)>0$. 
\end{thm}

\begin{proof}
The statement 1) follows from the definitions of the conditional correlations.
The statement 2) with degenerate $U$ (i.e., the unconditional version)
was proven by Rényi \cite{renyi1959measures}. The statement 2) with
non-degenerate $U$ (i.e., the conditional version) follows by combining
Rényi's result on the unconditional version \cite{renyi1959measures}
and the characterization of conditional maximal correlation in \eqref{eq:-4}. 
\end{proof}
\textcolor{blue}{Next we focus on the Gaussian case. It was shown
in \cite{ardestanizadeh2012linear} that the conditional maximal correlation
and the conditional Pearson correlation are equal for jointly Gaussian
random variables, i.e., $\rho_{\mathrm{m}}(X;Y|U)=|\rho(X;Y|U)|$.
However, in \cite{ardestanizadeh2012linear}, the conditional Pearson
correlation was defined differently, although it is equal to our definition
for the Gaussian case. More specifically, the conditional Pearson
correlation in \cite{ardestanizadeh2012linear} was defined as the
expectation of the event conditional correlation, i.e., $\mathbb{E}_{u}\rho(X;Y|U=u)$.
}
\begin{thm}
\label{thm:-(Gaussian-case).}\cite{ardestanizadeh2012linear} (Gaussian
case). For jointly Gaussian random variables $X,Y,U$, we have 
\begin{align}
 & |\rho(X;Y|U)|=\theta(X;Y|U)=\theta(Y;X|U)=\rho_{\mathrm{m}}(X;Y|U).\label{eq:-18}
\end{align}
\end{thm}

For completeness, we provide the following proof of Theorem \ref{thm:-(Gaussian-case).},
in which the properties derived above are applied.
\begin{proof}
The unconditional version of \eqref{eq:-18} was proven in \cite[Sec. IV, Lem. 10.2]{rozanov1967stationary}.
On the other hand, given $U=u,$ $(X,Y)$ also follows jointly Gaussian
distribution, and $\rho(X;Y|U=u)=\rho(X;Y|U)$ for any $u$. Hence
\begin{align}
\rho_{\mathrm{m}}(X;Y|U) & =\esssup_{u}\rho_{\mathrm{m}}(X;Y|U=u)\label{eq:-7}\\
 & =\esssup_{u}|\rho(X;Y|U=u)|\label{eq:-8}\\
 & =|\rho(X;Y|U)|,\nonumber 
\end{align}
where \eqref{eq:-7} follows from Theorem \ref{thm:Alternative-characterization},
and \eqref{eq:-8} follows from the unconditional version \cite[Sec. IV, Lem. 10.2]{rozanov1967stationary}.

Furthermore, both $\theta(X;Y|U)$ and $\theta(Y;X|U)$ are between
$\rho_{\mathrm{m}}(X;Y|U)$ and $|\rho(X;Y|U)|$. Hence \eqref{eq:-18}
holds. 
\end{proof}

\subsection{Other Properties: Tensorization, DPI, Correlation ratio equality,
and Conditioning reducing covariance gap}

The tensorization property and the data processing inequality for
the unconditional maximal correlation were proven in \cite[Thm. 1]{witsenhausen1975sequences}
and \cite[Lem. 2.1]{yu2008maximal} respectively. Here we extend them
to the conditional case. 
\begin{thm}
\label{lem:MCsequence}(Tensorization). Assume $(X^{n},Y^{n})=(X_{i},Y_{i})_{i=1}^{n}$
and given $U,$ $(X_{i},Y_{i}),1\le i\le n$ are conditionally independent.
Then we have 
\[
{\displaystyle \rho_{\mathrm{m}}(X^{n};Y^{n}|U)=\max_{1\leq i\leq n}\rho_{\mathrm{m}}(X_{i};Y_{i}|U)}.
\]
\textcolor{blue}{}
\end{thm}

\begin{proof}
The unconditional version 
\[
{\displaystyle \rho_{\mathrm{m}}(X^{n};Y^{n})=\max_{1\leq i\leq n}\rho_{\mathrm{m}}(X_{i};Y_{i})},
\]
for a sequence of pairs of independent random variables $(X^{n},Y^{n})$
is proven in \cite[Thm. 1]{witsenhausen1975sequences}. Hence the
result for the event conditional maximal correlation also holds. Using
this result and Theorem \ref{lem:Singular-value-characterization},
we have 
\begin{align}
{\displaystyle \rho_{\mathrm{m}}(X^{n};Y^{n}|U)} & =\esssup_{u}\rho_{\mathrm{m}}(X^{n};Y^{n}|U=u)\nonumber \\
 & =\esssup_{u}\max_{1\leq i\leq n}\rho_{\mathrm{m}}(X_{i};Y_{i}|U=u)\nonumber \\
 & ={\displaystyle \max_{1\leq i\leq n}\esssup_{u}\rho_{\mathrm{m}}(X_{i};Y_{i}|U=u)}\label{eq:-9}\\
 & ={\displaystyle \max_{1\leq i\leq n}\rho_{\mathrm{m}}(X_{i};Y_{i}|U)},\nonumber 
\end{align}
where \eqref{eq:-9} follows by the following lemma. 
\begin{lem}
Assume $\mathcal{I}$ is a countable set, and $\mathrm{P}_{U}$ is
an arbitrary distribution on $(\mathbb{R},\mathcal{B}_{\mathbb{R}})$.
Then for any function $f:\mathcal{I}\times\mathbb{R}\to\mathbb{R}$,
we have 
\[
\esssup_{u}\sup_{i\in\mathcal{I}}f(i,u)=\sup_{i\in\mathcal{I}}\esssup_{u}f(i,u).
\]
\end{lem}

This lemma follows from the following two points. For a number $\epsilon>0$,
assume $i^{*}\in\mathcal{I}$ satisfies that $\esssup_{u}f(i^{*},u)\geq\sup_{i\in\mathcal{I}}\esssup_{u}f(i,u)-\epsilon$.
Then for any function $f$, 
\[
\esssup_{u}\sup_{i\in\mathcal{I}}f(i,u)\geq\esssup_{u}f(i^{*},u)=\sup_{i\in\mathcal{I}}\esssup_{u}f(i,u)-\epsilon.
\]
Since $\epsilon>0$ is arbitrary, we have $\esssup_{u}\sup_{i\in\mathcal{I}}f(i,u)\geq\sup_{i\in\mathcal{I}}\esssup_{u}f(i,u)$.

On the other hand, denote $\lambda_{i}^{*}:=\esssup_{u}f(i,u)$. Then
by the definition of $\esssup$, we have $\mathrm{P}_{U}\left\{ u:f(i,u)>\lambda_{i}^{*}\right\} =0$
for all $i\in\mathcal{I}$. Hence by the union bound, we have $\mathrm{P}_{U}\left\{ u:\exists i\in\mathcal{I}\textrm{ s.t. }f(i,u)>\lambda_{i}^{*}\right\} =0$.
Furthermore, for a number $\epsilon>0$, $\sup_{i\in\mathcal{I}}f(i,u)>\sup_{i\in\mathcal{I}}\lambda_{i}^{*}+\epsilon$
implies that there exists an $i'\in\mathcal{I}$ such that $f(i',u)\geq\sup_{i\in\mathcal{I}}f(i,u)-\epsilon>\sup_{i\in\mathcal{I}}\lambda_{i}^{*}\geq\lambda_{i'}^{*}$.
Hence 
\begin{align*}
\mathrm{P}_{U}\left\{ u:\sup_{i\in\mathcal{I}}f(i,u)>\sup_{i\in\mathcal{I}}\lambda_{i}^{*}+\epsilon\right\}  & \leq\mathrm{P}_{U}\left\{ u:\exists i\in\mathcal{I}\textrm{ s.t. }f(i,u)>\lambda_{i}^{*}\right\} =0.
\end{align*}
By the definition of $\esssup$, we have $\esssup_{u}\sup_{i\in\mathcal{I}}f(i,u)\leq\sup_{i\in\mathcal{I}}\lambda_{i}^{*}+\epsilon$.
Since $\epsilon>0$ is arbitrary, we have $\esssup_{u}\sup_{i\in\mathcal{I}}f(i,u)\leq\sup_{i\in\mathcal{I}}\esssup_{u}f(i,u)$. 
\end{proof}
\begin{thm}
\label{thm:(Data-processing-inequality).}(Data processing inequality).
If random variables $X,Y,Z,U$ form a Markov chain $X\rightarrow(Z,U)\rightarrow Y$
(i.e., $X$ and $Y$ are conditionally independent given $(Z,U)$),
then 
\begin{align}
|\rho(X;Y|U)| & \leq\theta(X;Z|U)\theta(Y;Z|U),\label{eq:-19}\\
\theta(X;Y|U) & \leq\theta(X;Z|U)\rho_{\mathrm{m}}(Y;Z|U),\label{eq:-23}\\
\rho_{\mathrm{m}}(X;Y|U) & \leq\rho_{\mathrm{m}}(X;Z|U)\rho_{\mathrm{m}}(Y;Z|U).\label{eq:-20}
\end{align}
Moreover, equalities hold in \eqref{eq:-19}-\eqref{eq:-20} if $(X,Z,U)$
and $(Y,Z,U)$ have the same joint distribution. 
\end{thm}

\begin{proof}
Consider 
\begin{align}
\mathbb{E}[\mathrm{cov}(X,Y|U)] & =\mathbb{E}[(X-\mathbb{E}[X|U])(Y-\mathbb{E}[Y|U])]\nonumber \\
 & =\mathbb{E}[\mathbb{E}[(X-\mathbb{E}[X|U])(Y-\mathbb{E}[Y|U])|Z,U]]\nonumber \\
 & =\mathbb{E}[\mathbb{E}[X-\mathbb{E}[X|U]|Z,U]\mathbb{E}[Y-\mathbb{E}[Y|U]|Z,U]]\label{eq:-21}\\
 & =\mathbb{E}[(\mathbb{E}[X|Z,U]-\mathbb{E}[X|U])(\mathbb{E}[Y|Z,U]-\mathbb{E}[Y|U])]\nonumber \\
 & \leq\sqrt{\mathbb{E}[(\mathbb{E}[X|Z,U]-\mathbb{E}[X|U])^{2}]\mathbb{E}[(\mathbb{E}[Y|Z,U]-\mathbb{E}[Y|U])^{2}]}\label{eq:-22}\\
 & =\sqrt{\mathbb{E}[\mathrm{v}\mathrm{a}\mathrm{r}(\mathbb{E}[X|Z,U]|U)]\mathbb{E}[\mathrm{v}\mathrm{a}\mathrm{r}(\mathbb{E}[Y|Z,U]|U)]}\nonumber 
\end{align}
where \eqref{eq:-21} follows by the conditional independence, and
\eqref{eq:-22} follows by the Cauchy-Schwarz inequality. Hence 
\begin{align*}
|\rho(X;Y|U)| & =\frac{\mathbb{E}[\mathrm{cov}(X,Y|U)]}{\sqrt{\mathbb{E}[\mathrm{v}\mathrm{a}\mathrm{r}(X|U)]}\sqrt{\mathbb{E}[\mathrm{v}\mathrm{a}\mathrm{r}(Y|U)]}}\\
 & {\displaystyle \leq\sqrt{\frac{\mathbb{E}[\mathrm{v}\mathrm{a}\mathrm{r}(\mathbb{E}[X|Z,U]|U)]\mathbb{E}[\mathrm{v}\mathrm{a}\mathrm{r}(\mathbb{E}[Y|Z,U]|U)]}{\mathbb{E}[\mathrm{v}\mathrm{a}\mathrm{r}(X|U)]\mathbb{E}[\mathrm{v}\mathrm{a}\mathrm{r}(Y|U)]}}}\\
 & =\theta(X;Z|U)\theta(Y;Z|U).
\end{align*}

By the definitions of conditional correlation ratio and conditional
maximal correlation and using the relationships \eqref{eq:} and \eqref{eq:-27},
\eqref{eq:-23} and \eqref{eq:-20} can be derived from \eqref{eq:-19}.
Furthermore, it is easy to verify that equalities in \eqref{eq:-19}-\eqref{eq:-20}
hold if $(X,Z,U)$ and $(Y,Z,U)$ have the same joint distribution. 
\end{proof}
\begin{thm}
\label{lem:Correlation-ratio-equality}(Correlation ratio equality).
For any random variables $X,Y,Z,U,$ 
\begin{align}
1-\theta^{2}(X;Y,Z|U) & =(1-\theta^{2}(X;Z|U))(1-\theta^{2}(X;Y|Z,U));\label{eq:-15}\\
1-\rho_{\mathrm{m}}^{2}(X;Y,Z|U) & \geq(1-\rho_{\mathrm{m}}^{2}(X;Z|U))(1-\rho_{\mathrm{m}}^{2}(X;Y|Z,U));\label{eq:-14}\\
\theta(X;Y,Z|U) & \geq\theta(X;Y|Z,U);\label{eq:-16}\\
\rho_{\mathrm{m}}(X;Y,Z|U) & \geq\rho_{\mathrm{m}}(X;Y|Z,U).\label{eq:-51}
\end{align}
\end{thm}

\begin{rem}
The inequality in \eqref{eq:-51} is analogue to a similar property
for mutual information, i.e., $I(X;Y,Z|U)\geq I(X;Y|Z,U)$, where\footnote{If $X,Y,W$ are not discrete, then the term $\frac{P_{X,Y|W}(\cdot|w)}{P_{X|W}(\cdot|w)P_{Y|W}(\cdot|w)}$
in \eqref{eq:-10} is replaced by the Radon--Nikodym derivative of
$\mathrm{P}_{X,Y|W}(\cdot|w)$ with respect to $\mathrm{P}_{X|W}(\cdot|w)\mathrm{P}_{Y|W}(\cdot|w)$.} 
\begin{equation}
I(X;Y|W):=\mathbb{E}\log\frac{P_{X,Y|W}(X,Y|W)}{P_{X|W}(X|W)P_{Y|W}(Y|W)}\label{eq:-10}
\end{equation}
denotes the conditional mutual information between $X$ and $Y$ given
$W$ \cite{Cover}. If $W$ is degenerate, then $I(X;Y):=I(X;Y|W)$
is the (unconditional) mutual information between $X$ and $Y$. Furthermore,
$\rho_{\mathrm{m}}(X;Y|Z,U)\geq\rho_{\mathrm{m}}(X;Y|U)$ and $\rho_{\mathrm{m}}(X;Y|Z,U)\leq\rho_{\mathrm{m}}(X;Y|U)$
do not always hold. This is also analogue to the fact that $I(X;Y|Z,U)\geq I(X;Y|U)$
and $I(X;Y|Z,U)\leq I(X;Y|U)$ do not always hold. 
\end{rem}

\begin{proof}
From \eqref{eq:-12}, we have 
\[
1-{\displaystyle \theta^{2}(X;Y,Z|U)=\frac{\mathbb{E}[\mathrm{v}\mathrm{a}\mathrm{r}(X|Y,Z,U)]}{\mathbb{E}[\mathrm{v}\mathrm{a}\mathrm{r}(X|U)]}},
\]
\[
1-{\displaystyle \theta^{2}(X;Z|U)=\frac{\mathbb{E}[\mathrm{v}\mathrm{a}\mathrm{r}(X|Z,U)]}{\mathbb{E}[\mathrm{v}\mathrm{a}\mathrm{r}(X|U)]}},
\]
\[
1-{\displaystyle \theta^{2}(X;Y|Z,U)=\frac{\mathbb{E}[\mathrm{v}\mathrm{a}\mathrm{r}(X|Y,Z,U)]}{\mathbb{E}[\mathrm{v}\mathrm{a}\mathrm{r}(X|Z,U)]}}.
\]
Hence \eqref{eq:-15} follows immediately.

The inequality \eqref{eq:-14} follows since 
\begin{align*}
1-\rho_{\mathrm{m}}^{2}(X;Y,Z|U) & =\inf_{f}\left\{ 1-\theta^{2}(f(X,U);Y,Z|U)\right\} \\
 & =\inf_{f}\left\{ (1-\theta^{2}(f(X,U);Z|U))(1-\theta^{2}(f(X,U);Y|Z,U))\right\} \\
 & \geq\inf_{f}(1-\theta^{2}(f(X,U);Z|U))\inf_{f}(1-\theta^{2}(f(X,U);Y|Z,U))\\
 & =(1-\rho_{\mathrm{m}}^{2}(X;Z|U))(1-\rho_{\mathrm{m}}^{2}(X;Y|Z,U)).
\end{align*}

Furthermore, observe that $\theta^{2}(X;Z|U)\geq0$. Hence \eqref{eq:-16}
follows immediately from \eqref{eq:-15}.

By \eqref{eq:-16}, we have 
\begin{align*}
\rho_{\mathrm{m}}(X;Y|Z,U) & =\sup_{f}\theta(f(X,U);Y|Z,U)\leq\sup_{f}\theta(f(X,U);Y,Z|U)\leq\rho_{\mathrm{m}}(X;Y,Z|U).
\end{align*}
\end{proof}
\begin{cor}
\textcolor{blue}{\label{cor:For-any-random}}For any random variables
$U,X,Y,V$ such that ${\normalcolor }U\rightarrow X\rightarrow Y$
and $X\rightarrow Y\rightarrow V$, we have 
\begin{equation}
\rho_{\mathrm{m}}(U,X;V,Y)=\max\{\rho_{\mathrm{m}}(X;Y),\rho_{\mathrm{m}}(U;V|X,Y)\}.\label{eq:-2}
\end{equation}
\end{cor}

\begin{rem}
A ``dual'' (i.e., additivity) property holds for mutual information,
i.e., for any random variables $U,X,Y,V$ such that $U\rightarrow X\rightarrow Y$
and $X\rightarrow Y\rightarrow V$, 
\[
I(U,X;V,Y)=I(X;Y)+I(U;V|X,Y).
\]
\end{rem}

\begin{rem}
\textcolor{blue}{Corollary \ref{cor:For-any-random} can be applied
to measure the non-local correlations in a bipartite quantum system.
Imagine that two parties share (possibly correlated) subsystems of
a bipartite physical system. Each party applies a measurement on her
subsystem by tuning her measurement device according to some parameter,
and obtains a measurement outcome. Denote the measurement parameters
by $X$ and $Y$ , and the measurement outcomes by $U$ and $V$.
Assume the no-signaling condition holds, i.e., the random variables
$U,X,Y,V$ satisfy Markov chains ${\normalcolor }U\rightarrow X\rightarrow Y$
and $X\rightarrow Y\rightarrow V$. For this case, the conditional
distribution $\mathrm{P}_{U,V|X,Y}$ is termed a }\textcolor{blue}{\emph{no-signaling
box}}\textcolor{blue}{{} and the measurement parameter pair $\left(X,Y\right)$
is termed a }\textcolor{blue}{\emph{priori correlation.}}\textcolor{blue}{{}
 Hence by Corollary \ref{cor:For-any-random}, the equality \eqref{eq:-2}
holds for such a system. This means that the maximal correlation between
the two input-output pairs of the box, is equal to the maximum of
the priori maximal correlation between the two parties and the maximal
correlation of the box shared between them.  For more details, please
refer to \cite{beigi2015monotone}. }
\end{rem}

\begin{proof}
Beigi and Gobari \cite[Eqn. (4)]{beigi2015monotone} proved ${\displaystyle \rho_{\mathrm{m}}(U,X;V,Y)\leq\max\{\rho_{\mathrm{m}}(X;Y),\rho_{\mathrm{m}}(U;V|X,Y)\}}$.
Hence we only need to prove that ${\displaystyle \rho_{\mathrm{m}}(U,X;V,Y)\geq\max\{\rho_{\mathrm{m}}(X;Y),\rho_{\mathrm{m}}(U;V|X,Y)\}}$.
According to the definition, $\rho_{\mathrm{m}}(U,X;V,Y)\geq\rho_{\mathrm{m}}(X;Y)$
is straightforward. From \eqref{eq:-51} of Theorem \ref{lem:Correlation-ratio-equality},
we have $\rho_{\mathrm{m}}(U,X;V,Y)\geq\rho_{\mathrm{m}}(U,X;V|Y)\geq\rho_{\mathrm{m}}(U;V|X,Y)$.
This completes the proof.
\end{proof}
We also prove that conditioning reduces covariance gap as shown in
the following theorem, the proof of which is given in Appendix \ref{sec:Proof-of-Theorem-cond}. 
\begin{thm}
\label{lem:Conditioning-reduces-covariance}(Conditioning reduces
covariance gap). For any random variables $X,Y,Z,U,$ 
\begin{align*}
\sqrt{\mathbb{E}\mathrm{var}(X|Z,U)\mathbb{E}\mathrm{var}(Y|Z,U)}-\mathbb{E}\mathrm{cov}(X,Y|Z,U) & \leq\sqrt{\mathbb{E}\mathrm{var}(X|Z)\mathbb{E}\mathrm{var}(Y|Z)}-\mathbb{E}\mathrm{cov}(X,Y|Z),
\end{align*}
i.e., 
\begin{align*}
\sqrt{(1-\theta^{2}(X;U|Z))(1-\theta^{2}(Y;U|Z))}(1-\rho(X,Y|Z,U)) & \leq1-\rho(X,Y|Z).
\end{align*}
\end{thm}

\begin{rem}
The following two inequalities follow immediately. 
\begin{align*}
\sqrt{(1-\rho_{\mathrm{m}}^{2}(X;U|Z))(1-\theta^{2}(Y;U|Z))}(1-\theta(X,Y|Z,U)) & \leq1-\theta(X,Y|Z),\\
\sqrt{(1-\rho_{\mathrm{m}}^{2}(X;U|Z))(1-\rho_{\mathrm{m}}^{2}(Y;U|Z))}(1-\rho_{\mathrm{m}}(X,Y|Z,U)) & \leq1-\rho_{\mathrm{m}}(X,Y|Z).
\end{align*}
\end{rem}

\section{\textcolor{blue}{Application to Non-Interactive Simulation}}

\textcolor{blue}{Assume $\left(X,Y\right)\sim\mathrm{P}_{XY}$ is
a pair of random variables on a product measurable space $\left(\mathcal{X}\times\mathcal{Y},\mathcal{B}_{\mathcal{X}}\otimes\mathcal{B}_{\mathcal{Y}}\right)$,
and}\footnote{\textcolor{blue}{We use $\mathrm{P}_{XY}^{n}$ to denote the $n$-fold
product of distribution $\mathrm{P}_{XY}$ with itself.}}\textcolor{blue}{{} $(X^{n},Y^{n})\sim\mathrm{P}_{XY}^{n}$ are $n$
i.i.d. copies of $\left(X,Y\right)$. Then we focus on the following
non-interactive simulation problem: Given the distribution $\mathrm{P}_{XY}$
and a product measurable space $\left(\mathcal{U}\times\mathcal{V},\mathcal{B}_{\mathcal{U}}\otimes\mathcal{B}_{\mathcal{V}}\right)$,
what is the possible probability distribution $\mathrm{P}_{UV}$ on
$\left(\mathcal{U}\times\mathcal{V},\mathcal{B}_{\mathcal{U}}\otimes\mathcal{B}_{\mathcal{V}}\right)$
such that $U^{n}\to X^{n}\to Y^{n}\to V^{n}$ and $\left(U^{n},V^{n}\right)\sim\mathrm{P}_{UV}^{n}$?}
\begin{defn}
\textcolor{blue}{The }\textcolor{blue}{\emph{simulation set}}\textcolor{blue}{{}
of $n$ i.i.d. pairs $\left(X^{n},Y^{n}\right)\sim\mathrm{P}_{XY}^{n}$
is defined as 
\[
\mathcal{S}_{n}(\mathrm{P}_{XY}):=\left\{ \mathrm{P}_{UV}:\exists\left(U^{n},V^{n}\right)\sim\mathrm{P}_{UV}^{n}\textrm{ s.t. }U^{n}\to X^{n}\to Y^{n}\to V^{n}\right\} .
\]
}
\end{defn}

\textcolor{blue}{This problem is termed }\textcolor{blue}{\emph{Non-Interactive
Simulation of Random Variables }}\textcolor{blue}{\cite{kamath2016non}.
The case in which $X,Y,U,V\in\{0,1\}$ and only one-dimensional $\left(U,V\right)$
is required to be generated was studied in \cite{yu2019bounds}. The
non-interactive simulation problem is motivated naturally by several
models in distributed control systems and cryptography. It is also
related to the }\textcolor{blue}{\emph{Non-Interactive Correlation
Distillation Problem}}\textcolor{blue}{, in which the collision probability
$\mathbb{P}\left(U=V\right)$ is required to be maximized \cite{yang2007possibility,mossel2005coin,mossel2006non}.
Therefore, studying the non-interactive simulation problem is not
only of theoretical significance, but is also of tremendous applicabilities.
Furthermore, the non-interactive simulation problem or the non-interactive
correlation distillation problem can be also interpreted from the
perspectives of noise-stability (or noise-sensitivity); see \cite{mossel2005coin}.}

\textcolor{blue}{In this paper, we will provide an impossibility result
for the non-interactive simulation problem by using a new quantity
named }\textcolor{blue}{\emph{information-correlation function}}\textcolor{blue}{.
Before investigating the non-interactive simulation problem, we introduce
the information-correlation function first.}

\subsection{\textcolor{blue}{Information-Correlation Function}}
\begin{defn}
\textcolor{blue}{For $\left(X,Y\right)\sim\mathrm{P}_{XY}$, the }\textcolor{blue}{\emph{information-correlation
function}}\textcolor{blue}{{} of $X$ and $Y$ is defined by 
\begin{equation}
C_{\beta}(X;Y):={\displaystyle \inf_{\mathrm{P}_{W|X,Y}:\rho_{\mathrm{m}}(X;Y|W)\leq\beta}I(X,Y;W)},\beta\in[0,1],\label{eq:-24-3}
\end{equation}
where the mutual information $I(X,Y;W)$ is defined in \eqref{eq:-10}.
Furthermore, for $\beta\in(0,1]$, we define 
\begin{align*}
C_{\beta^{-}}(X;Y) & :=\lim_{\alpha\uparrow\beta}C_{\alpha}(X;Y).
\end{align*}
}
\end{defn}

\textcolor{blue}{Intuitively, the information-correlation function
of $X$ and $Y$ quantifies the minimum amount of ``common information''
that can be extracted from $\left(X,Y\right)$ such that the ``private
information'' of $X$ and $Y$ is at most $\beta$-correlated under
the conditional maximal correlation measure, i.e., $\rho_{\mathrm{m}}(X;Y|W)\leq\beta$.
Two closely related quantities are the Gács-Körner common information
\cite{gacs1973common} and Wyner common information \cite{Wyner}.
The former is defined in \eqref{eq:GK}, and the latter is defined
as follows. The Wyner common information between $X$ and $Y$ is
\[
C_{\mathrm{W}}(X;Y):=\inf_{\mathrm{P}_{W|X,Y}:X\to W\to Y}I(X,Y;W).
\]
Properties of the information-correlation function, as well as its
relationship to the Gács-Körner common information and Wyner common
information are shown in the following proposition. The proof is given
in Appendix \ref{sec:ICF}. }
\begin{prop}
\textup{\textcolor{blue}{\label{lem:ICFproperties}(a) If $\left(X,Y\right)$
has finite support $\mathcal{X}\times\mathcal{Y}$, then for the infimum
in \eqref{eq:-24-3}, it suffices to restrict the support size of
$W$ such that $|\mathcal{W}|\leq|\mathcal{X}||\mathcal{Y}|.$}}

\textup{\textcolor{blue}{(b) For any random variables $X,Y,$ the
information-correlation function $C_{\beta}(X;Y)$ is non-increasing
in $\beta.$ Moreover, 
\begin{align}
 & C_{\beta}(X;Y)=0\textrm{ for }\rho_{\mathrm{m}}(X;Y)\leq\beta\leq1,\label{eq:-60}\\
 & C_{\beta}(X;Y)>0\textrm{ for }0\leq\beta<\rho_{\mathrm{m}}(X;Y),\label{eq:-positive}\\
 & C_{0}(X;Y)=C_{\mathrm{W}}(X;Y),\label{eq:-61}\\
 & C_{1^{-}}(X;Y)=C_{\mathrm{GK}}(X;Y).\label{eq:-63}
\end{align}
}}

\textup{\textcolor{blue}{(c) If $P_{W|X,Y}$ attains the infimum in
\eqref{eq:-24-3}, then $\rho_{\mathrm{m}}(X;Y|W)\leq\rho_{\mathrm{m}}(X;Y|V)$
for any $V$ such that $\left(X,Y\right)\rightarrow W\rightarrow V$.}}

\textup{\textcolor{blue}{(d) (Additivity) Assume $(X^{n},Y^{n})$
are $n$ i.i.d. pairs of random variables. Then we have 
\begin{equation}
C_{\beta}(X^{n};Y^{n})={\displaystyle \sum_{i=1}^{n}C_{\beta}(X_{i};Y_{i})}.\label{eq:-64}
\end{equation}
}}
\end{prop}

\begin{rem}
\textcolor{blue}{For any pair of random variables $\left(X,Y\right)$,
$C_{\beta}(X;Y)$ is non-increasing in $\beta$, but it is not necessarily
convex or concave; see the Gaussian case in the next subsection. $C_{\beta}(X;Y)$
is discontinuous at $\beta=1,$ if the Gács-Körner common information
between $X,Y$ is strictly positive. Lemma \ref{lem:ICFproperties}
implies the Gács-Körner common information and Wyner common information
are two special points on the information-correlation function. }
\end{rem}

\textcolor{blue}{Another important property of the information-correlation
function --- the data processing inequality --- is provided in the
following proposition. }
\begin{prop}
\textup{\textcolor{blue}{\label{thm:DPI_ICF}(Data processing inequality).
If random variables $X,Z,Y$ form a Markov chain $X\rightarrow Z\rightarrow Y$
(i.e., $X$ and $Y$ are conditionally independent given $Z$), then
\[
C_{\beta}(X;Y)\leq C_{\beta}(X;Z),\forall\beta\in[0,1].
\]
}}
\end{prop}

\begin{proof}
\textcolor{blue}{Assume random variables $X,Z,Y$ form a Markov chain
$X\rightarrow Z\rightarrow Y$. For an arbitrary conditional distribution
$\mathrm{P}_{W|X,Z}$, we introduce a new random vector $W$ such
that $\left(X,Z,Y,W\right)\sim\mathrm{P}_{XZ}\mathrm{P}_{Y|Z}\mathrm{P}_{W|X,Z}$.
Hence $W\rightarrow(X,Z)\rightarrow Y$ and $X\rightarrow(Z,W)\rightarrow Y$.
By the data processing inequality on mutual information \cite[Theorem 2.8.1]{Cover},
we have 
\begin{equation}
I(X,Y;W)\leq I\left(X,Z;W\right).\label{eq:-31}
\end{equation}
By the data processing inequality on maximal correlation (Theorem
\ref{thm:(Data-processing-inequality).}), we have 
\begin{equation}
\rho_{\mathrm{m}}(X;Y|W)\leq\rho_{\mathrm{m}}\left(X;Z|W\right).\label{eq:-33}
\end{equation}
Combining \eqref{eq:-31} and \eqref{eq:-33}, we obtain that 
\[
C_{\beta}(X;Y)\leq C_{\beta}(X;Z),\forall\beta\in[0,1].
\]
}
\end{proof}
\textcolor{blue}{For jointly Gaussian random variables, the information-correlation
function is characterized in the following proposition. The proof
is given in Appendix \ref{sec:Gaussian}. }
\begin{prop}
\textup{\textcolor{blue}{\label{thm:Gaussian}(Gaussian random variables).
For jointly Gaussian random variables $\left(X,Y\right)$ with correlation
coefficient $\beta_{0},$ 
\begin{equation}
C_{\beta}(X;Y)=\frac{1}{2}\log^{+}\left[\frac{1+\beta_{0}}{1-\beta_{0}}/\frac{1+\beta}{1-\beta}\right].
\end{equation}
}}
\end{prop}

\begin{rem}
\textcolor{blue}{When specialized to the case $\beta=0$, we obtain
$C_{\mathrm{W}}(X;Y)=C_{0}(X;Y)={\displaystyle \frac{1}{2}\log^{+}\left[\frac{1+\beta_{0}}{1-\beta_{0}}\right]}$,
which was first proven in \cite{xu2013wyner}. }
\end{rem}

\textcolor{blue}{For the symmetric bivariate random variable, an upper
bound on the information-correlation function is given in the following
proposition. The proof is given in Appendix \ref{sec:DSBS}. }
\begin{prop}
\textup{\textcolor{blue}{\label{thm:DSBS} For the symmetric bivariate
random variable $\left(X,Y\right)$ with distribution 
\[
P_{XY}=\left[\begin{array}{cc}
\frac{1}{2}\left(1-p_{0}\right) & \frac{1}{2}p_{0}\\
\frac{1}{2}p_{0} & \frac{1}{2}\left(1-p_{0}\right)
\end{array}\right],
\]
(i.e., the crossover probability being $p_{0}$), we have 
\begin{equation}
C_{\beta}(X;Y)\leq1+H_{2}\left(p_{0}\right)-H_{4}\left(\frac{1}{2}\left(1-p_{0}+\sqrt{\frac{1-2p_{0}-\beta}{1-\beta}}\right),\frac{1}{2}\left(1-p_{0}-\sqrt{\frac{1-2p_{0}-\beta}{1-\beta}}\right),\frac{p_{0}}{2},\frac{p_{0}}{2}\right)\label{eq:-48}
\end{equation}
for $0\leq\beta<1-2p_{0}$, and $C_{\beta}(X;Y)=0$ for $\beta\geq1-2p_{0}$,
where 
\begin{align}
H_{2}(p) & =-p\log p-(1-p)\log(1-p),\label{eq:binaryentropy}\\
H_{4}(a,b,c,d) & =-a\log a-b\log b-c\log c-d\log d
\end{align}
respectively denote the binary and quaternary entropy functions. }}
\end{prop}

\begin{rem}
\textcolor{blue}{Numerical results show that the upper bound in \eqref{eq:-48}
is tight. }
\end{rem}

\subsection{\textcolor{blue}{Impossibility Result}}

\textcolor{blue}{Based on the information-correlation function, we
can establish the following impossibility result for the non-interactive
simulation problem. }
\begin{thm}
\textup{\textcolor{blue}{\label{thm:Impossibility}The simulation
set of $n$ i.i.d. pairs $\left(X^{n},Y^{n}\right)\sim\mathrm{P}_{XY}^{n}$
satisfies 
\[
\mathcal{S}_{n}(\mathrm{P}_{XY})\subseteq\left\{ \mathrm{P}_{UV}:C_{\beta}(U;V)\leq C_{\beta}(X;Y),\forall\beta\in[0,1]\right\} .
\]
}}
\end{thm}

\begin{proof}
\textcolor{blue}{Consider the simulation of $\left(U^{n},V^{n}\right)\sim\mathrm{P}_{UV}^{n}$
from $\left(X^{n},Y^{n}\right)\sim\mathrm{P}_{XY}^{n}$ such that
$U^{n}\to X^{n}\to Y^{n}\to V^{n}$. Applying the vector version of
the data processing inequality in Proposition \ref{thm:DPI_ICF},
we obtain that 
\[
C_{\rho}(U^{n};V^{n})\leq C_{\rho}(X^{n};V^{n})\leq C_{\rho}(X^{n};Y^{n}).
\]
By the additivity of the information-correlation function (Lemma \ref{lem:ICFproperties}),
we obtain 
\[
C_{\rho}(U;V)\leq C_{\rho}(X;Y).
\]
}
\end{proof}
\textcolor{blue}{Intuitively, since $X^{n},Y^{n},U^{n},V^{n}$ form
a Markov chain $U^{n}\to X^{n}\to Y^{n}\to V^{n}$, it makes sense
that $X^{n}$ and $Y^{n}$ possess more ``common information'' than
$U^{n}$ and $V^{n}$. Hence the necessary condition given in Theorem
\ref{thm:Impossibility} holds.}

\textcolor{blue}{We can also obtain the following simple bounds for
the simulation problem. By the data processing inequality on maximal
correlation (i.e., the unconditional version of Theorem \ref{thm:(Data-processing-inequality).}),
we obtain the following outer bound. 
\begin{equation}
\mathcal{S}_{n}(\mathrm{P}_{XY})\subseteq\left\{ \mathrm{P}_{UV}:\rho_{\mathrm{m}}(U;V)\leq\rho_{\mathrm{m}}(X;Y)\right\} .\label{eq:-36}
\end{equation}
From \eqref{eq:-positive}, we know that our outer bound in Theorem
\ref{thm:Impossibility} is at least as tight as the outer bound in
\eqref{eq:-36}. By the data processing inequality on mutual information
\cite[Theorem 2.8.1]{Cover}, we obtain the following outer bound.
\begin{equation}
\mathcal{S}_{n}(\mathrm{P}_{XY})\subseteq\left\{ \mathrm{P}_{UV}:I(U;V)\leq I(X;Y)\right\} .\label{eq:-38}
\end{equation}
Furthermore, we can obtain the following inner bound by using a pair
of product conditional distributions $(\mathrm{P}_{U|X}^{n},\mathrm{P}_{V|Y}^{n})$.
\begin{equation}
\mathcal{S}_{n}(\mathrm{P}_{XY})\supseteq\left\{ \mathrm{P}_{UV}:\exists(\mathrm{P}_{U|X},\mathrm{P}_{V|Y})\textrm{ s.t. }\mathrm{P}_{UV}\textrm{ is the marginal distribution of }\mathrm{P}_{U|X}\mathrm{P}_{X,Y}\mathrm{P}_{V|Y}\textrm{ on }(U,V)\right\} .\label{eq:-35}
\end{equation}
}

\textcolor{blue}{Our outer bound in Theorem \ref{thm:Impossibility},
the maximal correlation outer bound in \eqref{eq:-36}, the mutual
information outer bound in \eqref{eq:-38}, and the inner bound in
\eqref{eq:-35} are plotted in Fig. \ref{fig:simulation}. For this
figure, we assume $X,Y,U,V\in\{0,1\}$, i.e., they are binary random
variables. For this case, the joint distribution of $X$ and $Y$
is determined by the triple $\left(P_{X}(0),P_{Y}(0),P_{X,Y}(0,0)\right)$,
and so is the joint distribution of $U$ and $V$. For Fig. \ref{fig:simulation},
we assume $P_{X}(0)=P_{Y}(0)=\frac{1}{4}$ and $P_{X,Y}(0,0)=p$.
Similarly, we assume $P_{U}(0)=P_{V}(0)=\frac{1}{2}$ and $P_{U,V}(0,0)=q$.
Given $p$, we focus on the possible range of $q$. The end points
of the possible range of $q$ are located between our outer bound
and the inner bound.}

\begin{figure}[t]
\begin{centering}
\includegraphics[width=0.7\textwidth]{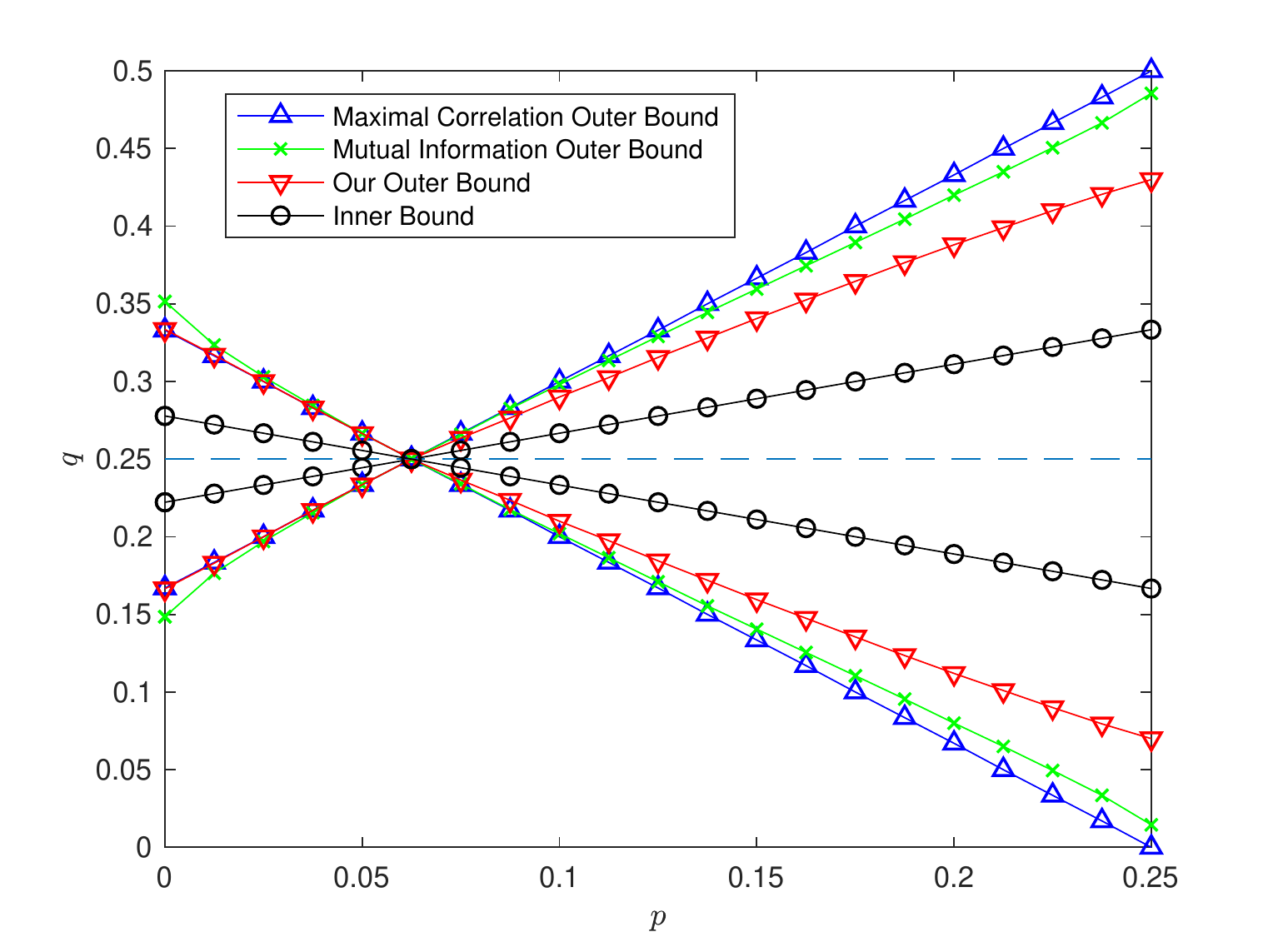} 
\par\end{centering}
\caption{\label{fig:simulation}Illustration of our outer bound in Theorem
\ref{thm:Impossibility}, the maximal correlation outer bound in \eqref{eq:-36},
the mutual information outer bound in \eqref{eq:-38}, and the inner
bound in \eqref{eq:-35}. The curves are symmetric with respect to
the line $q=\frac{1}{4}$. }
\end{figure}

\section{Concluding Remarks}

In this paper, we defined several conditional correlation measures
and derived their properties, especially for the conditional maximal
correlation. From these properties, one can observe that the maximal
correlation and correlation ratio share many similar properties as
the mutual information, such as invariance to bijections, chain rule
(correlation ratio equality), data processing inequality, etc. On
the other hand, the maximal correlation and correlation ratio also
have some properties that are different from those of the mutual information.
For example, for a sequence of pairs of independent random variables,
the mutual information between them is the sum of mutual information
over all pairs of components (i.e., additivity); while the maximal
correlation is the maximum of the maximal correlations over all pairs
of components (i.e., tensorization). Furthermore, we used the conditional
maximal correlation to define the information-correlation function,
and derived a data processing inequality for such a function. As an
application, we applied this data processing inequality to obtain
an impossibility result for the non-interactive simulation problem.

The (conditional) maximal correlation also has applications in inference
and privacy. In inference and privacy, a fundamental question is that:
Given an observation $Y$, how much information can we learn about
a hidden random variable $X$ from $Y$? Or equivalently, how much
information is leaked from $X$ to $Y$? In \cite{asoodeh2015maximal,calmon2013bounds,issa2018operational,li2018maximal,rassouli2017perfect},
the maximal correlation $\rho_{\mathrm{m}}(X;Y)$ was used to measure
the information leakage from $X$ to $Y$ (or from $Y$ to $X$).
\textcolor{blue}{Furthermore, in \cite{yu2016distortion}, the present
author, together with Li and Chen,}\textcolor{blue}{\emph{ }}\textcolor{blue}{applied
the conditional maximal correlation to derive converse results for
the problem of transmission of correlated sources over a multi-access
channel, in which the correlated sources are assumed to have a common
part. The tensorization and data processing properties of the conditional
maximal correlation (derived in the present paper) play a crucial
role in such an application.}

\appendix

\section{\label{sec:Proof-of-Theorem-cond}Proof of Theorem \ref{lem:Conditioning-reduces-covariance}}

For simplicity, we only prove the degenerate $Z$ case, i.e., 
\begin{align}
\sqrt{\mathbb{E}\mathrm{var}(X|U)\mathbb{E}\mathrm{var}(Y|U)}-\mathbb{E}\mathrm{cov}(X,Y|U) & \leq\sqrt{\mathrm{var}(X)\mathrm{var}(Y)}-\mathrm{cov}(X,Y).\label{eq:-17}
\end{align}
For non-degenerate $Z$ case, it can be proven similarly.

By the law of total covariance, we have 
\begin{align*}
\mathrm{cov}(X,Y) & =\mathbb{E}\mathrm{cov}(X,Y|U)+\mathrm{cov}(\mathbb{E}(X|U),\mathbb{E}(Y|U)).
\end{align*}
Hence to prove \eqref{eq:-17}, we only need to show 
\begin{align}
\sqrt{\mathbb{E}\mathrm{var}(X|U)\mathbb{E}\mathrm{var}(Y|U)}+\mathrm{cov}(\mathbb{E}(X|U),\mathbb{E}(Y|U)) & \leq\sqrt{\mathrm{var}(X)\mathrm{var}(Y)}.\label{eq:-34}
\end{align}

To prove this, we consider 
\begin{align}
\mathbb{E}\mathrm{var}(X|U)\mathbb{E}\mathrm{var}(Y|U)= & \left(\mathrm{var}(X)-\mathrm{var}(\mathbb{E}(X|U))\right)\left(\mathrm{var}(Y)-\mathrm{var}(\mathbb{E}(Y|U))\right)\label{eq:-69}\\
= & \mathrm{var}(X)\mathrm{var}(Y)+\mathrm{var}(\mathbb{E}(X|U))\mathrm{var}(\mathbb{E}(Y|U))\nonumber \\
 & -\mathrm{var}(X)\mathrm{var}(\mathbb{E}(Y|U))-\mathrm{var}(Y)\mathrm{var}(\mathbb{E}(X|U))\nonumber \\
\leq & \mathrm{var}(X)\mathrm{var}(Y)+\mathrm{var}(\mathbb{E}(X|U))\mathrm{var}(\mathbb{E}(Y|U))\nonumber \\
 & -2\sqrt{\mathrm{var}(X)\mathrm{var}(\mathbb{E}(Y|U))\cdot\mathrm{var}(Y)\mathrm{var}(\mathbb{E}(X|U))}\nonumber \\
= & \left(\sqrt{\mathrm{var}(X)\mathrm{var}(Y)}-\sqrt{\mathrm{var}(\mathbb{E}(X|U))\mathrm{var}(\mathbb{E}(Y|U))}\right)^{2}\label{eq:-71}
\end{align}
where \eqref{eq:-69} follows from the law of total variance 
\begin{equation}
\mathrm{var}(X)=\mathbb{E}\mathrm{var}(X|U)+\mathrm{var}(\mathbb{E}(X|U)).\label{eq:-70}
\end{equation}

Since $\mathbb{E}\mathrm{var}(X|U)\geq0$, from \eqref{eq:-70}, we
have $\mathrm{var}(\mathbb{E}(X|U))\leq\mathrm{var}(X).$ Similarly,
we have $\mathrm{var}(\mathbb{E}(Y|U))\leq\mathrm{var}(Y).$ Therefore,
\begin{equation}
\mathrm{var}(\mathbb{E}(X|U))\mathrm{var}(\mathbb{E}(Y|U))\leq\mathrm{var}(X)\mathbb{E}\mathrm{var}(Y).\label{eq:-37}
\end{equation}

Combining \eqref{eq:-71} and \eqref{eq:-37}, we have 
\begin{align*}
\sqrt{\mathbb{E}\mathrm{var}(X|U)\mathbb{E}\mathrm{var}(Y|U)} & \leq\sqrt{\mathrm{var}(X)\mathrm{var}(Y)}-\sqrt{\mathrm{var}(\mathbb{E}(X|U))\mathrm{var}(\mathbb{E}(Y|U))}.
\end{align*}
Furthermore, by the Cauchy-Schwarz inequality, it holds that 
\begin{align*}
|\mathrm{cov}(\mathbb{E}(X|U),\mathbb{E}(Y|U))| & =|\mathbb{E}\left[\left(\mathbb{E}(X|U)-\mathbb{E}(X)\right)\left(\mathbb{E}(Y|U)-\mathbb{E}(Y)\right)\right]|\\
 & \leq\sqrt{\mathbb{E}\left(\mathbb{E}(X|U)-\mathbb{E}(X)\right)^{2}\cdot\mathbb{E}\left(\mathbb{E}(Y|U)-\mathbb{E}(Y)\right)^{2}}\\
 & =\sqrt{\mathrm{var}(\mathbb{E}(X|U))\mathrm{var}(\mathbb{E}(Y|U))}.
\end{align*}
Therefore, 
\begin{align*}
\sqrt{\mathbb{E}\mathrm{var}(X|U)\mathbb{E}\mathrm{var}(Y|U)} & \leq\sqrt{\mathrm{var}(X)\mathrm{var}(Y)}-|\mathrm{cov}(\mathbb{E}(X|U),\mathbb{E}(Y|U))|\\
 & \leq\sqrt{\mathrm{var}(X)\mathrm{var}(Y)}-\mathrm{cov}(\mathbb{E}(X|U),\mathbb{E}(Y|U)).
\end{align*}
This is just the inequality \eqref{eq:-34}. Hence the proof is complete.

\section{\textcolor{blue}{\label{sec:ICF}Proof of Lemma \ref{lem:ICFproperties}}}

\textcolor{blue}{Proof of (a): To show (a), we only need to show that
for any random variable $W$, there always exists another random variable
$W'$ with support $\mathcal{W}'$ such that $|\mathcal{W}'|\leq|\mathcal{X}||\mathcal{Y}|,$
$\rho_{\mathrm{m}}(X;Y|W')\leq\rho_{\mathrm{m}}(X;Y|W)$, and $I(X,Y;W')=I(X,Y;W)$.
According to the support lemma \cite{ElGamal}, there exists a random
variable $W'$ with $\mathcal{W}'\subseteq\mathcal{W}$ and $|\mathcal{W}'|\leq|\mathcal{X}||\mathcal{Y}|$
such that 
\begin{align}
H(X,Y|W') & ={\displaystyle H(X,Y|W)},\\
P_{X,Y} & ={\displaystyle \sum_{w'}P_{W'}(w')P_{X,Y|W'}(\cdot|w')}.\label{eq:-57}
\end{align}
Since $\mathcal{W}'\subseteq\mathcal{W}$, we have $\rho_{\mathrm{m}}(X;Y|W')\leq\rho_{\mathrm{m}}(X;Y|W)$.
Furthermore, \eqref{eq:-57} implies that $H(X,Y)$ is also preserved.
Hence $I(X,Y;W)=I(X,Y;W').$ This completes the proof of (a).}

\textcolor{blue}{Proof of (b): \eqref{eq:-60} and \eqref{eq:-positive}
follow straightforwardly by definition. By definition and Lemma \ref{lem:relationship}
($\rho_{\mathrm{m}}(X;Y|W)=0$ if and only if $X\rightarrow W\rightarrow Y$),
we obtain \eqref{eq:-61}.}

\textcolor{blue}{Next we prove \eqref{eq:-63}. Assume $\left(f^{*},g^{*}\right)$
attains the Gács-Körner common information between $X$ and $Y$ (see
the definition in \eqref{eq:GK}). Set $W=f^{*}(X)$, then we have
\begin{align*}
 & \rho_{\mathrm{m}}(X;Y|W)<1,\\
 & I(X,Y;W)=H(f^{*}(X))=C_{\mathrm{GK}}(X;Y).
\end{align*}
Hence by definition, 
\begin{equation}
C_{1^{-}}(X;Y)\leq C_{\mathrm{GK}}(X;Y).\label{eq:-25}
\end{equation}
}

\textcolor{blue}{On the other hand, for any $W$ such that $\rho_{\mathrm{m}}(X;Y|W)<1$,
the random variable $f^{*}(X)$ is a deterministic function of $W$,
i.e., $f^{*}(X)=g(W)$ for some function $g$. This is because, otherwise,
we have 
\begin{align*}
\rho_{\mathrm{m}}(X;Y|W) & \geq\rho(f^{*}(X);g^{*}(Y)|W)\\
 & =\rho(f^{*}(X);f^{*}(X)|W)\\
 & =1.
\end{align*}
}

\textcolor{blue}{Since $f^{*}(X)$ is a deterministic function of
$W$, we have 
\[
I(X,Y;W)=I(X,Y;W,f^{*}(X))\geq H(f^{*}(X))=C_{\mathrm{GK}}(X;Y).
\]
Hence 
\begin{equation}
C_{1^{-}}(X;Y)\geq C_{\mathrm{GK}}(X;Y).\label{eq:-26}
\end{equation}
}

\textcolor{blue}{Combining \eqref{eq:-25} and \eqref{eq:-26} gives
us 
\[
C_{1^{-}}(X;Y)=C_{\mathrm{GK}}(X;Y).
\]
}

\textcolor{blue}{Proof of (c): Suppose $P_{W|X,Y}$ achieves the infimum
in \eqref{eq:-24-3}. If $V$ satisfies both $\left(X,Y\right)\rightarrow W\rightarrow V$
and $\left(X,Y\right)\rightarrow V\rightarrow W$, then we have $\rho_{\mathrm{m}}(X;Y|W)=\rho_{\mathrm{m}}(X;Y|W,V)=\rho_{\mathrm{m}}(X;Y|V)$.}

\textcolor{blue}{If $V$ satisfies $\left(X,Y\right)\rightarrow W\rightarrow V$
but does not satisfy $\left(X,Y\right)\rightarrow V\rightarrow W$,
then $I(X,Y;W)=I(X,Y;W,V)>I(X,Y;V)$. Hence $\rho_{\mathrm{m}}(X;Y|W)\leq\rho_{\mathrm{m}}(X;Y|V)$,
otherwise it contradicts with that $P_{WX,Y}$ achieves the infimum
in \eqref{eq:-24-3}.}

\textcolor{blue}{Proof of (d): For \eqref{eq:-64} it suffices to
prove the case of $n=2$, i.e., 
\begin{equation}
C_{\beta}(X^{2};Y^{2})=C_{\beta}(X_{1};Y_{1})+C_{\beta}(X_{2};Y_{2}).\label{eq:-66}
\end{equation}
}

\textcolor{blue}{Observe for any $P_{W|X^{2},Y^{2}}$, 
\[
\rho_{\mathrm{m}}(X^{2};Y^{2}|W)\geq\rho_{\mathrm{m}}(X_{i};Y_{i}|W),i=1,2,
\]
and 
\begin{align*}
I(X^{2},Y^{2};W) & \geq I(X_{1},Y_{1};W)+I(X_{2},Y_{2};W|X_{1},Y_{1})\\
 & =I(X_{1},Y_{1};W)+I(X_{2},Y_{2};W,X_{1},Y_{1})\\
 & \geq I(X_{1},Y_{1};W)+I(X_{2},Y_{2};W).
\end{align*}
Hence we have 
\begin{equation}
C_{\beta}(X^{2};Y^{2})\geq C_{\beta}(X_{1};Y_{1})+C_{\beta}(X_{2};Y_{2}).\label{eq:-65}
\end{equation}
}

\textcolor{blue}{Moreover, if we choose $P_{W|X^{2},Y^{2}}=P_{W_{1}|X_{1},Y_{1}}^{*}P_{W_{2}|X_{2},Y_{2}}^{*}$
in $C_{\beta}(X^{2};Y^{2})$, where $P_{W_{i}|X_{i},Y_{i}}^{*},i=1,2,$
is a distribution achieving $C_{\beta}(X_{i};Y_{i})$, then we have
\[
\rho_{\mathrm{m}}(X^{2};Y^{2}|W)=\max_{i\in\left\{ 1,2\right\} }\rho_{\mathrm{m}}(X_{i};Y_{i}|W_{i})\leq\beta,
\]
and 
\begin{equation}
I(X^{2},Y^{2};W)=I(X_{1},Y_{1};W_{1})+I(X_{2},Y_{2};W_{2})=C_{\beta}(X_{1};Y_{1})+C_{\beta}(X_{2};Y_{2}).\label{eq:-65-1}
\end{equation}
Therefore, 
\begin{equation}
C_{\beta}(X^{2};Y^{2})=\inf_{P_{W|X^{2},Y^{2}}:\rho_{\mathrm{m}}(X^{2};Y^{2}|W)\leq\beta}I(X^{2},Y^{2};W)\leq C_{\beta}(X_{1};Y_{1})+C_{\beta}(X_{2};Y_{2}).\label{eq:-66-1}
\end{equation}
}

\textcolor{blue}{Inequalities \eqref{eq:-65} and \eqref{eq:-66-1}
imply \eqref{eq:-64} with $n=2$.}

\section{\textcolor{blue}{Proof of Proposition \ref{thm:Gaussian}\label{sec:Gaussian}}}

\textcolor{blue}{For continuous random variables, a lower bound on
the information-correlation function is given in the following lemma. }
\begin{lem}
\textup{\textcolor{blue}{\label{thm:Lower-bound}(Lower bound on $C_{\beta}(X;Y)$).
For any absolutely continuous random variables $\left(X,Y\right)$
with correlation coefficient $\beta_{0},$ we have 
\begin{equation}
C_{\beta}(X;Y){\displaystyle \geq h(X,Y)-\frac{1}{2}\log\left[(2\pi e(1-\beta_{0}))^{2}\frac{1+\beta}{1-\beta}\right]}
\end{equation}
for $0\leq\beta\leq\beta_{0}$, and $C_{\beta}(X;Y)=0$ for $\beta_{0}\leq\beta\leq1.$ }}
\end{lem}

\begin{proof}
\textcolor{blue}{
\begin{align}
I(X,Y;W) & =h(X,Y)-h(X,Y|W)\nonumber \\
 & {\displaystyle \geq h(X,Y)-\mathbb{E}_{W}\frac{1}{2}\log\left[(2\pi e)^{2}\det(\Sigma_{XY|W})\right]}\label{eq:-11}\\
 & {\displaystyle \geq h(X,Y)-\frac{1}{2}\log\left[(2\pi e)^{2}\det(\mathbb{E}_{W}\Sigma_{XY|W})\right]}\label{eq:-27-1-2}\\
 & =h(X,Y)-{\displaystyle \frac{1}{2}\log\left[(2\pi e)^{2}\left[\mathbb{E}\textrm{var}(X|W)\mathbb{E}\textrm{var}(Y|W)-(\mathbb{E}\textrm{cov}(X,Y|W))^{2}\right]\right]}\nonumber \\
 & =h(X,Y)-{\displaystyle \frac{1}{2}\log\left[(2\pi e)^{2}\mathbb{E}\textrm{var}(X|W)\mathbb{E}\textrm{var}(Y|W)(1-\rho^{2}(X;Y|W))\right]}\nonumber \\
 & {\displaystyle \geq h(X,Y)-\frac{1}{2}\log\left[(2\pi e)^{2}(\frac{1-\beta_{0}}{1-\rho(X,Y|W)})^{2}(1-\rho^{2}(X;Y|W))\right]}\label{eq:-28-1-2}\\
 & =h(X,Y)-{\displaystyle \frac{1}{2}\log\left[(2\pi e(1-\beta_{0}))^{2}\frac{1+\rho(X;Y|W)}{1-\rho(X;Y|W)}\right]}\nonumber \\
 & {\displaystyle \geq h(X,Y)-\frac{1}{2}\log\left[(2\pi e(1-\beta_{0}))^{2}\frac{1+\beta}{1-\beta}\right]},\label{eq:-29-1-2}
\end{align}
where \eqref{eq:-11} follows from the fact that given the covariance
matrix $\Sigma_{XY}$ of $\left(X,Y\right)$, $h(X,Y)\leq\frac{1}{2}\log\left[(2\pi e)^{2}\det(\Sigma_{XY})\right]$,
\eqref{eq:-27-1-2} follows from the function $\log(\det(\cdot))$
is concave on the set of symmetric positive definite square matrices
\cite[p.73]{boyd2004convex}, \eqref{eq:-28-1-2} follows from Lemma
\ref{lem:Conditioning-reduces-covariance}, and \eqref{eq:-29-1-2}
follows from the constraint $\rho(X;Y|W)\leq\beta.$ }
\end{proof}
\textcolor{blue}{Furthermore, it is easy to verify that equality in
Theorem \ref{thm:Lower-bound} holds if $X,Y$ are jointly Gaussian.
(This can be shown by choosing $W$ such that $(W,X,Y)$ are jointly
Gaussian.) Hence we complete the proof.}

\section{\textcolor{blue}{Proof of Proposition \ref{thm:DSBS}\label{sec:DSBS}}}

\textcolor{blue}{Assume $W\sim\textrm{Bern}(\frac{1}{2})$. Define
two distributions $P_{X,Y|W}(\cdot|0)=\left[\begin{array}{cc}
a & p_{0}/2\\
p_{0}/2 & b
\end{array}\right],P_{X,Y|W}(\cdot|1)=\left[\begin{array}{cc}
b & p_{0}/2\\
p_{0}/2 & a
\end{array}\right]$ with $a+b=1-p_{0}$. Then $P_{X,Y}=P_{W}(0)P_{X,Y|W}(\cdot|0)+P_{W}(1)P_{X,Y|W}(\cdot|1)$.
By using the formula \cite{anantharam2014hypercontractivity} 
\[
\rho_{\mathrm{m}}^{2}(X;Y)=\left[\sum_{x,y}\frac{P^{2}(x,y)}{P(x)P(y)}\right]-1
\]
for binary-valued $(X,Y)$, we have 
\[
\rho_{\mathrm{m}}(X;Y)=1-2p_{0}.
\]
and 
\begin{align}
\rho_{\mathrm{m}}(X;Y|W & =0)=\rho_{\mathrm{m}}(X;Y|W=1)\nonumber \\
 & =\sqrt{\frac{2\left(p_{0}/2\right)^{2}}{\left(a+p_{0}/2\right)\left(b+p_{0}/2\right)}+\frac{a^{2}}{\left(a+p_{0}/2\right)^{2}}+\frac{b^{2}}{\left(b+p_{0}/2\right)^{2}}-1}.\label{eq:-39}
\end{align}
Hence $\rho_{\mathrm{m}}(X;Y|W)$ is also equal to the RHS of \eqref{eq:-39}.
By choosing 
\begin{align*}
a & =\frac{1}{2}\left(1-p_{0}+\sqrt{\frac{1-2p_{0}-\beta}{1-\beta}}\right)\\
b & =\frac{1}{2}\left(1-p_{0}-\sqrt{\frac{1-2p_{0}-\beta}{1-\beta}}\right),
\end{align*}
we have $\rho_{\mathrm{m}}(X;Y|W)\leq\beta$. For this case, $I(X,Y;W)$
is equal to the RHS of \eqref{eq:-48}. Hence, by definition, the
RHS of \eqref{eq:-48} is an upper bound of $C_{\beta}(X;Y)$.}

\subsection*{Acknowledgments}

The authors are supported in part by a Singapore National Research
Foundation (NRF) National Cybersecurity R\&D Grant (R-263-000-C74-281
and NRF2015NCR-NCR003-006), and in part by a National Natural Science
Foundation of China (NSFC) under Grant (61631017).

 \bibliographystyle{unsrt}
\bibliography{ref}

\end{document}